\documentclass[reqno, 12pt]{amsproc}

\usepackage[utf8]{inputenc}
\usepackage{amsmath}
\usepackage[margin=1in]{geometry}
\usepackage{mathtools}
\usepackage{amssymb}
\usepackage{extarrows}
\usepackage{amsthm}
\usepackage{amsmath,mathabx}
\usepackage{mathtools}
\usepackage{cases}
\usepackage{booktabs}
\usepackage{blkarray}
\usepackage[english]{babel}
\usepackage{booktabs}
\usepackage[linesnumbered,ruled,vlined]{algorithm2e}
\SetKwInput{KwInput}{Input}                
\SetKwInput{KwOutput}{Output}              
\usepackage[hidelinks,colorlinks=true,linkcolor=blue,citecolor=blue]{hyperref}

\usepackage[all,cmtip]{xy}

\usepackage{pgfplots,mathpazo}
\pgfplotsset{compat=1.17}
\usepackage{mathrsfs}
\usetikzlibrary{arrows}

\theoremstyle{plain}
\newtheorem{theorem}{Theorem}[section]
\newtheorem{prop}[theorem]{Proposition}
\newtheorem{lemma}[theorem]{Lemma}
\newtheorem{cor}[theorem]{Corollary}
\theoremstyle{remark}
\newtheorem{rem}[theorem]{Remark}
\newtheorem{ex}[theorem]{Example}

\theoremstyle{definition}
\newtheorem{definition}[theorem]{Definition}

\newcommand{\GL}{\text{\rm GL}}

\newcommand{\Stab}{\text{\rm Stab}}
\newcommand{\F}{\mathbb{F}}
\newcommand{\Z}{\mathbb{Z}}
\newcommand{\set}[1]{\left\{#1\right\}}
\newcommand{\Int}{\text{\bf I}}
\newcommand{\Barc}{{\text{\bf Bar}}}
\newcommand{\B}{\mathscr{B}}
\newcommand{\bR}{\mathbf{R}}
\newcommand{\bC}{\mathbf{C}}

\title{The Space of Barcode Bases for Persistence Modules}
\author{Emile Jacquard, Vidit Nanda and Ulrike Tillmann}
\date{}

\begin{document}
	\maketitle
	\begin{abstract}
	    The barcode of a persistence module serves as a complete combinatorial invariant of its isomorphism class. Barcodes are typically extracted by performing changes of basis on a persistence module until the constituent matrices have a special form. Here we describe a new algorithm for computing barcodes which also keeps track of, and outputs, such a change of basis. Our main result is an explicit characterisation of the group of transformations that sends one barcode basis to another. Armed with knowledge of the entire space of barcode bases, we are able to show that any map of persistence modules can be represented via a partial matching between bars provided that neither source nor target admits nested bars in its barcode. We also generalise the algorithm and results described above to work for zizag modules.
	\end{abstract}
	
	\section*{Introduction}
	
	Persistence modules appear in different forms and guises across many areas of mathematics. In recent years, particular interest and focus has come from their extensive use in topological data analysis (TDA) in general and persistent homology in particular. In this paper, we examine persistence modules of the form $(V_\bullet, f_\bullet)$: 	
    \begin{align} \label{equ:persitencemodule}
	\xymatrixcolsep{.5in}
	\xymatrix{
		V_{0} \ar@{->}[r]^-{f_1} & V_{1} \ar@{->}[r]^-{f_2} & \cdots  \ar@{->}[r]^-{f_{\ell-1}} & V_{\ell-1} \ar@{->}[r]^-{f_\ell} & V_{\ell},
	}
    \end{align}
where each  $V_i$ is a vector space of finite dimension $n_i$ (over an underlying field $\mathbb{F}$) 
and each $f_i:V_{i-1} \to V_i$ is a linear map. We study three different aspects. 

\subsection*{Computing Barcode Bases} The first is the central question of { finding a barcode basis} for $(V_\bullet,f_\bullet)$. This amounts to a choice of basis for each $V_i$ with respect to which the linear maps $f_i$ have a particularly nice form --- they admit at most a single $1$ in each row and column, with all other entries being $0$. The existence of such bases and matrix representations is well known \cite{gabriel}. We say that the matrices are in {\it barcode form} and  the corresponding  basis of $\bigoplus_{i=0}^\ell V_i$ is a {\it barcode basis}, since the barcodes familiar from TDA can easily be extracted. 

Algorithms to compute barcode  bases in TDA typically take as input a filtered chain complex as in \cite{comppers}, where one has recourse to matrix representations of the boundary operators. Algorithms for general persistence modules include the well-known \cite{zizag} and much more recently, \cite{hang2021correspondence} and  \cite{henselmanpetrusek2020matroids}. Here we present a new algorithm that takes as its input a matrix representation $A_\bullet = ( A_1, \dots , A_\ell)$ relative to some initial basis of the persistence module  $(V_\bullet, f_\bullet)$  and outputs a sequence $g= (g_0 , \dots, g_\ell)$ of change of basis matrices $g_i$ for each of the $V_i$ so that the new matrix representation $A'_\bullet = (A'_1, \dots , A'_\ell)$ with  $A_i' = g_i \cdot A_i \cdot g_{i-1}^{-1}$ is in barcode form. Our algorithm in Section \ref{subsection:algorithms} is explicit and elementary in the sense that every intermediate step amounts to performing standard (row or column) operations on the constituent $A_i$'s. The key difficulty here is that column operations on $A_i$ often force new matrix operations on $A_k$ for $k<i$, and similarly row operations on $A_i$ often require changes in $A_k$ for $k > i$. 

\subsection*{The Space of Barcode Bases} Our second goal is to {describe the set  of all barcode bases}\footnote{This set  has a natural topology when working over a field with topology such as $\mathbb R $ or $\mathbb C$.} of $(V_\bullet,f_\bullet)$. We show that this set can naturally be identified as the stabiliser of a  matrix representation $A_\bullet$ of $(V_\bullet, f_\bullet)$, and hence as  a subgroup of the product 
\[
G := \GL (n_0; \mathbb F)  \times \dots \times \GL (n_\ell; \mathbb F),
\] 
where $\GL(d;\F)$ indicates the general linear group of invertible $d \times d$ matrices with entries in $\F$. Writing 
$$
[i_1,j_1] \preceq [i_2,j_2] \quad \text{ whenever } \quad i_1 \leq i_2 \leq j_1 \leq j_2
$$
and $\text{Mat}(m \times n; \mathbb{F})$ for  the set of 
$m \times n$ matrices with coefficients in $\mathbb{F}$,
we show in Theorem \ref{theorem:stabbij} that the {set of barcode bases} is in one-to-one correspondence with
\[
\prod_{0 \leq i \leq j \leq \ell}  \GL(d_{ij};\mathbb{F})\times  \prod_{\substack{[i_1,j_1] \precneq [i_2,j_2] }} \text{\rm Mat}(d_{i_{1}j_{1} }\times d_{i_{2}j_{2}};\mathbb{F}).
\]
Here $d_{ij} $ is the multiplicity of the interval $[i,j]$ in the barcode of $(V_\bullet,f_\bullet)$. If the matrix representation is already in barcode form, then the elements in $\GL(d_{ij};\mathbb{F})$ correspond to changes of basis for the sub-vector space of $V_i$ spanned by the basis elements corresponding to the  bars $[i,j]$, and the elements in  $\text{\rm Mat}(d_{i_{1}j_{1} }\times d_{i_{2}j_{2}};\mathbb{F})$ represent the changes to basis vectors in $V_{i_2}$ corresponding to intervals $[i_2, j_2]$ obtained by adding vectors from $V_{i_2}$ which correspond to intervals $[i_1, j_1]$.

\subsection*{Simplifying Maps of Persistence Modules} We now turn attention to our third problem.  Given a map of persistence modules 
$$
\phi: (V_\bullet, f_\bullet) \to  (W_\bullet, h_\bullet),
$$
we seek { barcode bases for source and target in terms of which $\phi$ assumes its simplest form}, in the sense that we now specify. An interval  $[i,j]$ represents a submodule canonically isomorphic  to
the interval module $\Int  [i,j]_\bullet$. It is an elementary observation that such a module can be mapped non-trivially to another interval module $\Int [i', j']_\bullet$  if and only if $[i',j'] \preceq [i,j]$; and  in this case the non-zero map is unique up to a non-zero scalar.
In the simplest case, $\phi$ induces a partial matching where each bar in the source is mapped to exactly one in the target or mapped to zero. Surprisingly, we show in Theorem \ref{theorem:nonest} that such a partial matching exists (after a change of barcode bases) whenever neither source nor target admit a pair of strictly nested intervals in their respective barcode decompositions.\footnote{ Two bars
$[i, j]$ and $[i', j']$  are strictly 
nested if  $i<i'$ and $j'<j$.} In an example we also show that these conditions are necessary.

\subsection*{Zigzag Modules} Finally, we generalise the algorithm and theorems described above to {zizag modules} of a fixed type $\tau$: the linear maps of the persistence module \eqref{equ:persitencemodule} can go either forward $V_{i-1} \xrightarrow{f_i} V_{i}$ or backward $V_{i-1} \xleftarrow{q_i} V_{i}$ according to pattern fixed by $\tau$. Such modules are also classified in terms of sums of interval modules.
Our algorithm can be adapted to compute barcode bases of zizag modules of any type. 
Next, we introduce a generalisation of the order $\preceq$ that depends on the {type} $\tau$. This order takes into account that the order $\preceq$ has to be reversed when all the arrows in \eqref{equ:persitencemodule} are reversed.
With this order in place, we can once again classify the set of all barcode bases, see Theorem \ref{theorem:main}. Similarly when  considering maps of zizag modules, we  generalise the notion of strictly nested bars, which once again depends on the type $\tau$. Excluding such nested bars, we are able to obtain barcode bases of the source and target zizag modules in terms of which the map is described by a partial matching on the set of bars; see Theorem \ref{theorem:mainII}.

\subsection*{Outline}
We take the view that a persistence module $(V_\bullet, f_\bullet)$ is a quiver representation. In Section \ref{sec:barbasis}, we provide a constructive proof (and concomitant algorithm) of Gabriel's decomposition theorem for persistence modules. In Section \ref{sec:stabiliser} we identify the set of all barcode bases with a stabiliser of the action of the group $G$ on the set of all possible matrix representations of $(V_\bullet,f_\bullet)$. In Section \ref{sec:maps} we study maps between persistence modules by viewing them as representations of ladder quivers with relations. Gabriel's theorem no longer applies here; however, using similar arguments as in \cite{asashiba}, we are able to prove a finite decomposition when source and target have no nested bars. Finally, in Section \ref{sec:zigzag} we extend our results to any quiver of type A, that is zigzag persistence modules.

\subsection*{Related work} As mentioned above, there are several well known algorithms which compute barcodes of persistence modules. Some start with chains on a filtered simplicial complex, some deal with more general persistence modules and in some cases zizag modules. 

There are two algorithms that explicitly deal with computing the barcode bases associated to the interval decomposition. In \cite{carlsson2021persistent} the authors use matrix factorisation techniques to obtain bases in which the matrices are in echelon form. This technique also applies to zizag modules. In \cite{parallel} the authors inductively compute interval bases using basis completions techniques at each step, but they do not deal with the zizag case. Neither of these papers attempts to compute the set of barcode bases associated with the persistence module, and the algorithm we describe here takes a different approach to reducing matrices in barcode form.
Most recently, the authors of \cite{umatch} use U-match matrix factorisation to reduce computational complexity and memory storage in computing barcodes.

It was shown in \cite{escolar2015persistence} that maps between persistence modules of length less than 5 admit a tractable classification in the sense that the associated { ladder persistence modules} are always of finite type. In contrast, the authors of \cite{buchet2018realizations}  find an infinite class of indecomposable non-isomorphic ladder persistence module whenever the length is  greater then 5. In \cite{asashiba} the authors outline an algorithm which  computes the decomposition into a sum of indecomposables for ladder persistence modules of length $ < 5$.

\subsection*{Acknowledgement}
The authors are   members of the Centre for Topological Data Analysis, funded by the EPSRC grant EP/R018472/1. We are grateful to the two anonymous referees for their helpful comments and corrections.

	\section{Persistence Modules and Barcode Bases} \label{sec:setting}
	
	A \textbf{persistence module}, for the purposes of this paper, is a finite collection $(V_\bullet,f_\bullet)$ of finite-dimensional vector spaces $V_i$ over a field $\F$ along with $\F$-linear maps $f_i$ arranged as follows:
	\[
	\xymatrixcolsep{.5in}
	\xymatrix{
		V_{0} \ar@{->}[r]^-{f_1} & V_{1} \ar@{->}[r]^-{f_2} & \cdots  \ar@{->}[r]^-{f_{\ell-1}} & V_{\ell-1} \ar@{->}[r]^-{f_\ell} & V_{\ell}.
	}
	\]
	The number $\ell+1$ is called the {\em length} of $(V_\bullet,f_\bullet)$. The {\em  direct sum} of $(V_\bullet,f_\bullet)$ with another persistence module $(W_\bullet,h_\bullet)$ of the same length is defined pointwise --- in other words, the vector space at its $i$-th position is $V_i \oplus W_i$ for each admissible index $i$, and similarly the corresponding linear map is given by $f_i \oplus h_i$. We call $(V_\bullet,f_\bullet)$ {\em isomorphic} to $(W_\bullet,h_\bullet)$ if there are invertible linear maps $\phi_i:V_i \to W_i$ so that the square
	\[
	\xymatrixcolsep{.7in}
	\xymatrixrowsep{.4in}
	\xymatrix{
	V_{i-1} \ar@{->}[r]^{f_i} \ar@{->}[d]_{\phi_{i-1}}^{\sim} & V_{i} \ar@{->}[d]^{\phi_{i}}_{\sim} \\
	W_{i-1} \ar@{->}[r]_{h_i} & W_i 
	}
	\] 
	commutes for each index $i$ in $\set{1,\ldots,\ell}$. Isomorphisms from $(V_\bullet,f_\bullet)$ to itself are called {\em automorphisms}, and these evidently form a group under composition. We denote this group by $\text{Aut} (V_\bullet, f_\bullet)$.
	
	The {\bf interval module} corresponding to a pair of non-negative integers $i \leq j$ is the persistence module $\Int[i,j]_\bullet$ given by
	\[
	\xymatrixcolsep{.17in}
	\xymatrix{
		0 \ar@{->}[r] & \cdots \ar@{->}[r]  & 0 \ar@{->}[r] & \F \ar@{->}[r] & \cdots \ar@{->}[r]  & \F \ar@{->}[r] & 0 \ar@{->}[r] & \cdots \ar@{->}[r] & 0,
	}
	\]
	where the contiguous string of $\F$'s spans $\set{i,i+1,\ldots,j-1,j}$, all intermediate $\F\to\F$ maps are identities, and all other vector spaces are trivial. The importance of interval modules stems from the following result \cite{comppers}.
	
	\begin{theorem}\label{thm:indec}
		For each persistence module $(V_\bullet,f_\bullet)$ of length $\ell+1$, there exists a finite set of non-negative integer pairs
		\[\Barc(V_\bullet,f_\bullet) := \set{i_1 \leq j_1, \ldots, i_k \leq j_k},
		\] (with $[i_p,j_p] \subset [0,\ell]$  for all $1 \leq p \leq k$), called the {\bf barcode} of $(V_\bullet,f_\bullet)$, and an integer {\bf multiplicity} $d_{i_pj_p} > 0$ so that $(V_\bullet,f_\bullet)$ is isomorphic to a direct sum of interval modules:
		\begin{align} \label{eq:decomp}
		(V_\bullet,f_\bullet) \simeq \bigoplus_{p=1}^k \Int[i_p,j_p]_\bullet^{d_{i_pj_p}}.
		\end{align}
		Here the $i$-th summand on the right side is to be interpreted as the $d_{i_pj_p}$-fold direct sum of the interval module $\Int[i_p,j_p]_\bullet$ with itself.
	\end{theorem}

	This {\em interval decomposition} theorem follows from Gabriel's foundational result on the decomposability of quiver representations \cite{gabriel} --- since $(V_\bullet,f_\bullet)$ is a representation of a type-${\bf A}_{\ell+1}$ quiver. Our goal here is to provide an explicit algorithm which not only furnishes such the isomorphism \eqref{eq:decomp}, but can also be readily implemented on a computer.

	To this end, fix a persistence module $(V_\bullet,f_\bullet)$ of length $\ell+1$ and set $n_i := \dim_\F V_i$ for each $i \in \set{0,\ldots,\ell}$. Without loss of generality, we may select a {\em basis family}
	\[
	\B := \set{B_i \subset V_i \mid 0 \leq i \leq \ell},
	\] where each $B_i$ forms an ordered basis for the  vector space $V_i$.
	This choice amounts to fixing an isomorphism $V_i \simeq \F^{n_i}$ for each $i$. Thus, every linear map $f_i:V_{i-1} \to V_i$ can be represented (in terms of the chosen bases $B_{i-1}$ and $B_i$ from $\B$) as a matrix $A_i$ of size $n_i \times n_{i-1}$ with entries in $\F$; consequently, $(V_\bullet,f_\bullet)$ is isomorphic to
	\begin{align}\label{eq:basemod}
	\xymatrixcolsep{.5in}
	\xymatrix{
		\F^{n_0} \ar@{->}[r]^-{A_1} & \F^{n_1} \ar@{->}[r]^-{A_2} & \cdots  \ar@{->}[r]^-{A_{\ell-1}} & \F^{n_{\ell-1}} \ar@{->}[r]^-{A_\ell} & \F^{n_\ell}.
	}
	\end{align}
	In light of Theorem \ref{thm:indec}, we are particularly interested in a special class of basis families.	
	\begin{definition}\label{def:barbase}
			An $m \times n$ matrix $A$ of rank $r$ is in \textbf{barcode form} if there exists a strictly increasing function $c:\set{1,\ldots,r} \to \set{1,\ldots,n}$ so that
	\[
	A_{ij} = \begin{cases}
	                1 & \text{if } j = c(i), \\
	                0 & \text{otherwise.}
	            \end{cases}
	\]
	\end{definition}
		
Thus, a matrix is in barcode form whenever its entries lie in $\set{0,1}$, with at most one non-zero term in each row and column, and the $r$ non-zero terms appear in the first $r$ rows and in strictly increasing column order.

A basis family $\B$ is called an {\bf barcode basis} for $(V_\bullet,f_\bullet)$ if all of the $A_i$ are in barcode form. The natural basis arising from an interval decomposition of a persistence module is a barcode basis.

\begin{ex}
Consider, for instance, the persistence module of length $4$ given by the barcode containing $0 \leq 3$ along with $0 \leq 1$ and $1 \leq 3$, each with multiplicity one:
		
		
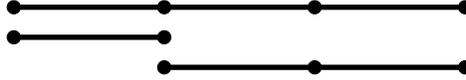
\begin{tikzpicture}[line cap=round,line join=round,>=triangle 45,x=1cm,y=.4cm]
		\clip(-8.853170731707314,2) rectangle (8.551707317073165,5.56487804878048);
		\draw [line width=2pt] (-4,5)-- (-2,5);
		\draw [line width=2pt] (-2,5)-- (0,5);
		\draw [line width=2pt] (0,5)-- (2,5);
		\draw [line width=2pt] (-4,4)-- (-2,4);
		\draw [line width=2pt] (-2,3)-- (0,3);
		\draw [line width=2pt] (0,3)-- (2,3);
		\begin{scriptsize}
		\draw [fill=black] (-4,5) circle (2.5pt);
		\draw [fill=black] (-2,5) circle (2.5pt);
		\draw [fill=black] (0,5) circle (2.5pt);
		\draw [fill=black] (2,5) circle (2.5pt);
		\draw [fill=black] (-4,4) circle (2.5pt);
		\draw [fill=black] (-2,4) circle (2.5pt);
		\draw [fill=black] (-2,3) circle (2.5pt);
		\draw [fill=black] (0,3) circle (2.5pt);
		\draw [fill=black] (2,3) circle (2.5pt);
		\end{scriptsize}
		\end{tikzpicture}
		With respect to the basis family obtained by ordering these  intervals from top to bottom, the matrices $A_i$ are given by
		\[
		A_{1}= \begin{blockarray}{cc}
		\begin{block}{[cc]}
		1 & 0 \\ 
		0 & 1 \\ 
		0 & 0  \\
		\end{block}  
		\end{blockarray}
		 \quad A_{2}= \begin{blockarray}{ccc}
		\begin{block}{[ccc]}
		1 & 0 & 0 \\ 
		0 & 0 &1 \\ 
		\end{block}  
		\end{blockarray}, \quad A_{3} = \begin{blockarray}{cc}
		\begin{block}{[cc]}
		1 & 0 \\ 
		0 & 1 \\ 
		\end{block}  
		\end{blockarray} ,
		\]
		and all three are evidently in barcode form. Conversely, one can also recover the interval decomposition immediately from these three matrices.
		\end{ex}
	
	To put our quest for a constructive proof of Theorem \ref{thm:indec} on a firm algebraic footing, let $X$ be the set of all the possible matrix-sequences $A_\bullet$ which can arise in \eqref{eq:basemod}. It is a (strict) subset of the product of matrices of the appropriate dimensions:
	\begin{align}\label{eq:X}
	X \subset \prod_{i=1}^\ell \text{Mat}\left(n_i \times n_{i-1};\F\right).
	\end{align}
	Writing $\GL(n;\F)$ for the group of all $n \times n$ invertible matrices over $\F$, consider the product
	\begin{align}\label{eq:G}
	G := \prod_{i=0}^\ell \GL(n_i;\F),
	\end{align}
	which acts naturally via a change-of-basis action on $X$: the group element $g := (g_0,\ldots,g_\ell)$ sends each matrix-sequence $A_\bullet$ in $X$ to the new sequence $(gA)_\bullet$ given by
	\begin{align}\label{eq:ga}
	(gA)_i := g_i \cdot A_i \cdot g_{i-1}^{-1}
	\end{align} 
    for each admissible index $i$. This is equivalent  to replacing the original basis family $\B = \set{B_i}$ with the new basis  family $g\B = \set{g_i B_i}$. Thus, $X$ is the free orbit of $A_\bullet$ under this $G$-action.  So our first task, solved in Section \ref{sec:barbasis}, translates to discovering some $g \in G$ that transforms a given basis family $\B$ of $(V_\bullet,f_\bullet)$ to a barcode basis. 
    
	\section{Constructing a Barcode Basis}\label{sec:barbasis}
	
	 Throughout this section, we fix a persistence module $(V_\bullet,f_\bullet)$ expressed as a sequence of matrices $A_\bullet$ as in \eqref{eq:basemod} with respect to an arbitrary (i.e., not necessarily barcode) basis family $\B$.

	\subsection{Barcode bases via elementary matrix operations}
	
	To conveniently describe relevant elements of $G$, we fix notation for matrices which implement certain fundamental row and column operations.
	
	\begin{definition}\label{def:emat}
		For each dimension $n > 0$, distinct indices $1 \leq p,q \leq n$, and scalar $\lambda \in \F$, let $e_{p,q}^n(\lambda)$ denote the {\bf elementary matrix} in $\GL(n;\F)$ which has $1$'s all along its diagonal, $\lambda$ in the $(p,q)$-th position, and zeros everywhere else. 
	\end{definition}
	Since the dimension $n$ will be clear from context, we omit it from the superscript and simply write $e_{p,q}(\lambda)$ to indicate the relevant elementary matrix. The following standard facts about such matrices will be freely used in the sequel ---
	\begin{enumerate}
		\item multiplying a matrix on the left by $e_{p,q}(\lambda)$ implements the following {\bf elementary row operation} 
		\[
		\text{Row}(p) \gets \text{Row}(p) + \lambda \cdot \text{Row}(q)
		\]
		which we denote $\bR_{p \gets q}(\lambda)$; similarly,
		\item multiplying a matrix on the right by $e_{p,q}(\lambda)$ implements the following {\bf elementary column operation}
		\[
		\text{Col}(q) \gets \text{Col}(q) + \lambda \cdot \text{Col}(p),
		\]
		which we denote $\bC_{q \gets p}(\lambda)$; and finally,
		\item the inverse of $e_{p,q}(\lambda)$ is $e_{p,q}(-\lambda)$.
	\end{enumerate}

	\begin{rem}\label{rem:seqops} Consider the element $g = (g_0,\ldots,g_\ell) \in G$ for which $g_i = e_{p,q}(\lambda)$ and all the other $g_j$ are identity matrices. The action of this $g$ on a given matrix sequence $A_\bullet$ is to simultaneously perform $\bC_{p \gets q}(\lambda)$ on $A_i$ and $\bR_{q \gets p}(-\lambda)$ on $A_{i-1}$ while leaving all the other $A_j$'s invariant.
	\end{rem}

	The following result plays an essential part in our constructive proof of Theorem \ref{thm:indec}. In its statement and beyond, we will use $A(p,q)$ to indicate the entry in the $p$-th row and $q$-th column of a given matrix $A$.

	\begin{lemma}\label{lem:matop}
		Assume that the first $\ell-1$ matrices $\set{A_i \mid 1 \leq i < \ell}$ of \eqref{eq:basemod} are in barcode form, and that the last matrix $A_\ell$ has a pivot in the $(r,q)$ position, i.e., $A_{\ell}(r,q)=1$ and all other entries in the $q$-th column are zero. If there is a nonzero entry $\alpha := A_\ell(r,p)$ in the same row $r$ but some other column $p > q$, then there exists $g \in  G$ with $g_{\ell}=\text{\rm Id}$ so that $(gA)_\bullet$ equals $A_\bullet$ except $A_\ell$ where the $\alpha$ entry is replaced by zero.
	\end{lemma}
	\begin{proof} We proceed by induction on $\ell$, noting that the case $\ell=1$ is immediately true since there is only one matrix in sight. Assume that the statement holds up to $\ell-1$. The $r$-th row of $A_\ell$ contains a pivot $1$ in the $q$-th column and some $\alpha \neq 0$ in the $p$-th column. To eliminate this offending $\alpha$, we perform $\bC_{p \gets q}(-\alpha)$ on $A_\ell$ by performing the basis change $e_{p,q}(-\alpha)$ on $V_{\ell-1}$. Since $A_\ell(r,q)$ is assumed to be a pivot, the only resulting difference in $A_\ell$ is that the $(r,p)$-th entry changes from $\alpha$ to $0$. But by Remark \ref{rem:seqops}, we are also compelled to perform $\bR_{q \gets p}(\alpha)$ on the preceding matrix $A_{\ell-1}$. This results in a new matrix $A'_{\ell-1}$, and there are now 2 cases to consider, of which only the second requires the inductive hypothesis:
		
	{\bf Case 1:} if the $p$-th row of $A_{\ell-1}$ is identically zero, then our row operation has had no effect whatsoever; thus, $A'_{\ell-1} = A_{\ell-1}$ is still in barcode form and we have arrived at the desired result.

	\textbf{Case 2}: If the $p$-th row of $A_{\ell-1}$ is nonzero,  then since $q<p$ and $A_{\ell-1}$ is in barcode form, we see that the $q$-th row of $A_{\ell-1}$ must also be non-zero.  Then by Definition \ref{def:barbase} they must have pivot ones in distinct columns, say {$c$ and $d$ respectively}, and furthermore $c < d$. Thus, after we have performed $\bR_{q \gets p}(\alpha)$ on $A_{\ell-1}$, the resulting matrix $A'_{\ell-1}$ has the form 
	\[A'_{\ell-1} = ~
	\begin{blockarray}{cccccccc}
	& c & & & & d\\
	\begin{block}{[ccccccc]c}
	0 & 1 & 0 & \cdots &0 & \alpha & 0 & q \\
	& 0& & & & 0 \\
	& \vdots & & & & \vdots \\
	& 0& &  & & 0\\
	0& 0&0 & \cdots& &1 & 0 & p  \\
	\end{block}
	\end{blockarray}
	\]
	By induction, there exists a $g \in \prod_{i=0}^{\ell-1} \GL_{n_i}(\mathbb{F})$ with $g_{\ell-1}=\text{\rm Id}$ so that $g_{i}A_{i}g^{-1}_{i-1}$ is still in barcode form for $1 \leq i \leq \ell-2$, and 
	\[g_{\ell-1}A_{\ell-1}g^{-1}_{\ell-2}=
		\begin{blockarray}{cccccccc}
	& c & & & & d\\
	\begin{block}{[ccccccc]c}
	0 & 1 & 0 & \cdots &0 & 0 & 0 & q \\
	& 0& & & & 0 \\
	& \vdots & & & & \vdots \\
	& 0& &  & & 0\\
	0& 0&0 & \cdots& &1 & 0 & p  \\
	\end{block}
	\end{blockarray}\] is again in barcode form. Furthermore, since $g_{\ell-1}=\text{\rm Id}$, the matrix $A_{\ell}$ is left unchanged by this change of basis. The desired basis change is $(g_0,g_1, \dots, g_{\ell-2},e_{p,q}(-\alpha),\text{\rm Id})$.
	\end{proof}

\begin{prop}\label{prop:gabcons}
Given the sequence of matrices $A_\bullet$ as in (6),  
there is a $g \in G$ such that $(gA)_\bullet$ has all its matrices in barcode form.
\end{prop}
\begin{proof}
When $\ell=1$, we may diagonalise the matrix $A_{1}$ via standard row and column operations. Proceeding by induction for $\ell > 1$, assume the existence of some group element
\[
g' = (g_0, \dots, g_{\ell-1}) \in \prod_{i=0}^{\ell-1}\GL_{n_{i}}(\mathbb{F})
\] satisfying the following property: the matrices
$g_{i}A_{i}g_{i-1}^{-1}$ are in barcode form for $1 \leq i \leq \ell-1$. 

Consider $g = (g',\text{Id}_{n_{\ell}})$, which evidently lies in $G$. Replacing $A_\bullet$ by $(gA)_\bullet$ if necessary, we may assume that $A_\bullet$ has its first $\ell-1$ matrices in barcode form. Performing row operations on $A_{\ell}$ has no impact on the previous matrices, as it corresponds to multiplying $A_{\ell}$ on the left by some $g_{\ell}$. Thus, we may  assume without loss of generality that all previous matrices are in barcode form while $A_\ell$ itself is in reduced row echelon form. By Lemma \ref{lem:matop}, there is a basis change $g \in G$ which zeroes out each non-pivot entry whilst maintaining the barcode form of the previous matrices. Applying these basis changes gets us to the desired barcode basis.
\end{proof}

 \begin{rem} If $\B$ is the basis family with respect to which $(V_\bullet, f_\bullet)$ has matrix form $A_\bullet$, then $g\B$ is a barcode basis where $g \in G$ is as in Proposition \ref{prop:gabcons}. We may therefore regard it as a constructive analogue of Theorem \ref{thm:indec}.
 \end{rem}
 
	\subsection{Algorithms}\label{subsection:algorithms} Here we describe algorithms which implement the constructions of Lemma \ref{lem:matop} and Proposition \ref{prop:gabcons}. In particular, the main algorithm \text{CompPers} described below accepts as input an initial sequence of matrices $A_\bullet$ as in \eqref{eq:basemod} and puts them in barcode form. The sub-computations which we require frequently have been isolated into concomitant subroutines, described as follows.
	
	 \begin{enumerate}
	     \item The first subroutine \textbf{ColOp} implements the inductive strategy underlying our proof of Lemma \ref{lem:matop}; in particular, this algorithm acts as step $k$ of the inductive procedure described in the proof of that lemma.
	 
	    \item The second subroutine $\textbf{Reduce}$ takes as input a sequence $A_\bullet$ for which the first $\ell-1$ matrices are in barcode form together with an invertible matrix $g \in \prod_{i=0}^{\ell-1}\GL(n_{i};\mathbb{F})$. It then reduces the final matrix $A_\ell$ until it is in barcode form, {\em while maintaining the barcode form of all previous matrices} and suitably updating the basis change $g$.
	 	
	    \item Finally, the main algorithm $\textbf{CompPers}(A_\bullet)$  takes as input an arbitrary sequence of matrices $A_\bullet$ and produces as output  $g \in G$ together with $(gA)_\bullet$ in barcode form. From these matrices we can directly access all intervals in barcode of $(V_\bullet,f_\bullet)$.
	  \end{enumerate}
	
	\begin{algorithm}
		\SetAlgoLined
		\KwInput{$A_\bullet$, $g$, $k$, $r$, $q$, $p$}
		\KwOutput{Updated $A_\bullet$ and basis change $g$, zeroing out $A_{k}(r,p)$}  
		$\bC_{p \gets q}(-\alpha)$ on $A_k$ \\
				$\bR_{q \gets p}(\alpha)$ on $A_{k-1}$ \\
		$g_{k}=e_{p,q}(-\alpha)$ \\
		\If{$k=0$ or $p$-th row of $A_{k-1}=0$}{ 
			return ($A_\bullet,g$)  
		}
		
		\Else{
			Find pivot columns $c<d$ of the pivot rows $q<p$ of $A_{k-1}$ \\
			return (\textbf{ColOp} ($A_\bullet,g,k-1,q,c,d$) )
		}
		\caption{\textbf{ColOp} }
		
	\end{algorithm}

	\begin{algorithm}
		\SetAlgoLined
		\KwInput{$A_\bullet$, $g$, where $A_\bullet$ has its first $\ell-1$ matrices in reduced form}
		\KwOutput{$A_\bullet$ in reduced barcode form and updated $g$}  
		row reduce ($A_{\ell})$ \\
		Append $g$ with corresponding $g_{\ell}$\\
		\While{there are $A_{\ell}(r,p) \neq 0$ terms with $A_{\ell}(r,q)=1$  a pivot}
		{$A_\bullet,h=$\textbf{ColOp}$(A_\bullet, g, \ell, r, q, p)$   \\
			$g=hg$ }
		
		return ($A_\bullet,g$)
		
		\caption{\textbf{Reduce} }
	\end{algorithm}

	\begin{algorithm}
	\label{alg:compers}

		\SetAlgoLined
		\KwInput{$A_\bullet$}
		\KwOutput{Reduced $A_\bullet$ with corresponding change of basis $g$}  
		$A'_\bullet=\left\lbrace A_{1} \right\rbrace$\\
		$g=(\text{\rm Id}_{n_{0}})$ \\
		\For{ $1 \leq i \leq \ell-1$}{
			$A'_\bullet,g=$\textbf{Reduce}$(A_\bullet',g)$\\
			$A_\bullet'$.append($A_{i+1}g_{i}^{-1}$) 
		}

		return (\textbf{Reduce}($A'_\bullet,g$))
		
		\caption{\textbf{CompPers}} 
	\end{algorithm}  

\begin{rem}
	The computational complexity of \textbf{CompPers}($A_\bullet$) can be expressed in terms of $n=\max_{ 0 \leq i \leq \ell} n_{i}$ and $\ell$. The cost of placing all the $A_i$ in reduced row echelon form via Gaussian elimination is $O(n^{3}\ell)$. Furthermore, performing column operations to further reduce these matrices requires at most $O(n^{2})$ operations on each matrix. And column operations on $A_{i}$ will, in the worst case, require down-stream column operations on $A_{i-1} \dots A_{1}$. Thus, for column operations, we have a 
	$O(n^{2} \sum _{i=1}^{\ell} i)=O(n^{2}\frac{\ell^(\ell-1)}{2})=O(n^{2}\ell^{2})$
	complexity. Combining these factors, the total complexity of the algorithm is 
	\[
	O(n^{3}\ell+n^{2}\ell^{2}).
	\]
    At each step of the algorithm, we perform an elementary basis change on a single vector space $V_i$, which amounts to multiplying a matrix $g_i \in \GL(V_i)$ by an elementary matrix., This incurs an $O(n)$ cost; thus, if we also wish to keep track of the basis changes, then the total complexity of {\textbf{CompPers}} becomes 
	 \[
	O(n^{4}\ell+n^{3}\ell^{2}).
	\]
\end{rem}	

We conclude with an illustrative example of how {\bf CompPers} acts on a sequence of input matrices.
	
	\begin{ex}
 Consider 	\[A_\bullet=
     \left(\left[\begin{matrix}
		1 & 0 & 0\\ 
		0 & 1 & 0\\ 
		0 & 0 & 0
		\end{matrix}\right], 
		\left[\begin{matrix}
		1 & 0 & 0\\ 
		0 & 1 & 0\\ 
		0 & 0 & 1
		\end{matrix}\right],
		\left[\begin{matrix}
		1 & 0 & 1\\ 
		0 & 1 & 1\\ 
		0 & 0 & 0
		\end{matrix}\right]\right)
		\]
	To put $A_3$ in barcode form, we must zero out the terms $A_{3}(1,3)$ and $A_{3}(2,3)$ using column operations. At each step, we will be performing row and column operations on matrices of $A_\bullet$, amounting to basis changes on the vectors spaces $V_{0}, V_{1}$ and $V_{2}$. For conciseness sake, we will not keep track of the basis changes done along the way, and will simply be performing operations on the matrices to put them in barcode form. 
	
	We begin by zeroing out the $A_{3}(2,3)$ term.
	
	\textbf{1}: $\bC_{3 \gets 2}(-1)$ on $A_3$, inducing  $\bR_{2 \gets 3}(1)$ on $A_2$, giving us matrices 
	\[
		 \begin{blockarray}{ccc}
		\begin{block}{[ccc]}
		1 & 0 & 0\\ 
		0 & 1 & 0\\ 
		0 & 0 & 0 \\
		\end{block}  
		\end{blockarray}, \;
		 \begin{blockarray}{ccc}
		\begin{block}{[ccc]}
		1 & 0 & 0 \\ 
		0 & 1 &1 \\
		0 & 0 &1 \\
		\end{block}  
		\end{blockarray}, \;
		\begin{blockarray}{ccc}
		\begin{block}{[ccc]}
		1 & 0 & 1 \\ 
		0 & 1 & 0\\
		0 & 0 & 0\\
		\end{block}  
		\end{blockarray} 
		\]
		
	\textbf{2}:  $\bC_{3 \gets 2}(-1)$ on $A_2$, inducing  $\bR_{2 \gets 3}(1)$. We see here that the third row of $A_{1}$ is zero, so we are in \textbf{Case 1} of Lemma \ref{lem:matop}, and so we are done.
	
	We have achieved our goal of zeroing out $A_{3}(2,3)$ whilst keeping the previous matrices in barcode form, making no other changes to $A_{3}$. It remains to zero out the $A_{3}(1,3)$ term.
	
	\textbf{1}:  $\bC_{3 \gets 1}(-1)$ on $A_3$, inducing  $\bR_{1 \gets 3}(1)$ on $A_2$ giving us matrices 
	\[
		 \begin{blockarray}{ccc}
		\begin{block}{[ccc]}
		1 & 0 & 0\\ 
		0 & 1 & 0\\ 
		0 & 0 & 0 \\
		\end{block}  
		\end{blockarray}, \;
		 \begin{blockarray}{ccc}
		\begin{block}{[ccc]}
		1 & 0 & 1 \\ 
		0 & 1 &0 \\
		0 & 0 &1 \\
		\end{block}  
		\end{blockarray}, \;
		\begin{blockarray}{ccc}
		\begin{block}{[ccc]}
		1 & 0 & 0 \\ 
		0 & 1 & 0\\
		0 & 0 & 0\\
		\end{block}  
		\end{blockarray} 
		\]\\
	
	\textbf{2:} $\bC_{3 \gets 1}(-1)$ on $A_2$, inducing  $\bR_{1 \gets 3}(1)$ on $A_1$. Since the third row of $A_{1}$ is zero, this operations has no impact on $A_{1}$, giving us matrices 
\[
	 \begin{blockarray}{ccc}
		\begin{block}{[ccc]}
		1 & 0 & 0\\ 
		0 & 1 & 0\\ 
		0 & 0 & 0 \\
		\end{block}  
		\end{blockarray}, \;
		 \begin{blockarray}{ccc}
		\begin{block}{[ccc]}
		1 & 0 & 0 \\ 
		0 & 1 &0 \\
		0 & 0 &1 \\
		\end{block}  
		\end{blockarray}, \;
		\begin{blockarray}{ccc}
		\begin{block}{[ccc]}
		1 & 0 & 0 \\ 
		0 & 1 & 0\\
		0 & 0 & 0\\
		\end{block}  
		\end{blockarray} 
		\]
	and so we are done. 

	\end{ex}
	
	\bigskip
	
	\section{The Space of Barcode Bases}\label{sec:stabiliser}
	
	Let $A_\bullet \in X$ be a sequence of matrices as in \eqref{eq:basemod} arising from an arbitrary choice of  basis for some persistence module $(V_\bullet,f_\bullet)$. Consider the group $G$ from  \eqref{eq:G}, recalling that $G$ acts on $X$ via change of basis. Our quest to describe all possible barcode bases for $(V_\bullet,f_\bullet)$ begins with a formula for the stabiliser of the chosen matrices $A_\bullet$ under this $G$-action. Namely, we seek the subgroup of $G$ given by
	\begin{align}\label{eq:stabA}
	\Stab(A_\bullet) := \set{g \in G \mid g_i\cdot A_i \cdot g_{i-1}^{-1} = A_i \text{ for all } 1 \leq i \leq \ell}.
	\end{align} To describe $\Stab(A_\bullet)$, we employ two binary relations on the set of all intervals which might possibly arise in the barcode decomposition of $(V_\bullet,f_\bullet)$ \`a la Theorem \ref{thm:indec}. 
	
	\begin{definition}\label{def:binrel}
	 Let $\preceq$ be the binary relation on $\set{[i,j] \in \Z^2 \mid  0 \leq i \leq j \leq \ell}$ given by
	  \[
	            [a,b] \preceq [c,d] \text{ whenever } a \leq c \leq b \leq d.
	 \]
	 (Although this relation $\preceq$ is reflexive and anti-symmetric on its domain, it is not transitive and hence does not form a partial order.)
	\end{definition}
	The second binary relation is the standard lexicographic order.
	\begin{definition}\label{def:lexord}
	        Let $\trianglelefteq$ be the lexicographic ordering on $\set{[i,j] \in \Z^2 \mid  0 \leq i \leq j \leq \ell}$, given by 
	        \begin{center}
	        $ [a,b] \trianglelefteq [c,d] \quad \Longleftrightarrow \quad a < c $ or $a=c$ and $b \leq d$.
	        \end{center}
	        	\end{definition}
    This yields a total order on the set of all possible bars in the interval decomposition of our persistence module $(V_\bullet, f_\bullet)$. 
    
    \begin{rem} The binary relation $\preceq$ is compatible with the lexicographical order $\trianglelefteq$ in the sense that $[a,b] \preceq [c,d]$ implies $[a,b] \trianglelefteq [c,d]$. 
    \end{rem}
    
    Given a barcode basis $\B$, we may totally order its bars using the lexicographic order, arbitrarily ordering bars with the same start and end point. This in turn yields a natural ordering of the bases $B_i$. The matrix representation of such bases is unique. Indeed, the first $d_{0i}$ basis vectors of $B_i$ are part of $[0,i]$ bars, then the next $d_{0,i+1}$ basis vectors are those part an $[0,i+1]$ bar, and so on following the lexicographic ordering until finally the $d_{i,\ell}$ basis vectors part of an $[i,\ell]$ bar. This yields a matrix representation for which $A_i$ is of the form
    \begin{align} \label{eq:obform}
        A_i=
    \begin{blockarray}{ccccccc}
    \begin{block}{[ccccccc]}
       M_0 & & & & &\\
       & M_1 & & & &\\
       & & \ddots & &\\
       & & & & M_{i}\\
    \end{block}
    \end{blockarray},
    \end{align}
    
    where
    \[M_j=\begin{blockarray}{ccccccc}
    \begin{block}{[ccccccc]}
    & &\text{Id}_{d_{j,i}} & & & &\\
    &0& & \text{Id}_{d_{j,i+1}} & & & \\
    && & & \ddots & & \\
    && & & & & \text{Id}_{d_{j,\ell}}\\
    \end{block}
    \end{blockarray}
    \]
    is a matrix of dimension $(\sum_{k=i}^{\ell}d_{jk}) \times (\sum_{k=i-1}^{\ell} d_{jk})$. As such, bases that have been ordered in the above way are barcode bases in the usual sense; we call these {\bf ordered} barcode bases of $(V_\bullet,f_\bullet)$ and devote the remainder of this section to completely characterising them.

	\begin{theorem}\label{theorem:stabbij}
	For each pair $[i,j]$ in $\set{0,1,\ldots,\ell}$ with $i \leq j$, let $d_{ij}$ equal the multiplicity of $i \leq j$ in the barcode of $(V_\bullet,f_\bullet)$, with the understanding that $d_{ij} = 0$ whenever $[i,j]$ is not in $\Barc(V_\bullet,f_\bullet)$. Then there is a bijection of sets:
	\[
	\Stab(A_\bullet) ~ \cong ~ \prod_{[i,j]}\GL(d_{ij};\mathbb{F})\times \hspace{-.22in} \prod_{[i_1,j_1] \precneq [i_2,j_2] } \hspace{-.2in} \text{\rm Mat}(d_{i_{1}j_{1} }\times d_{i_{2}j_{2}};\mathbb{F}).
	\]
	(The induced group structure on the right side is given in Corollary \ref{cor:stabmult} below)
	\end{theorem}
	\begin{proof}
		Elements in the same orbit have isomorphic stabilisers, so without loss of generality we may assume $A_\bullet$ is given by the matrix representation of the linear maps in a ordered barcode basis. An element $g = (g_0,\ldots,g_\ell)$ of $G$ lies in $\Stab(A_\bullet)$ if and only if we have an equality of matrix products
		\[
		g_i\cdot A_i  = A_i \cdot g_{i-1}
		\]
		for each $i \in \set{1,\ldots,\ell}$. Set $k_i := \text{rank }A_i$ and note that since $A_i$ is in barcode form, there is a strictly increasing function $c_i:\set{1,\ldots,k_i} \to \set{1,\ldots,n_{i-1}}$ so that the unique nonzero entry in the $p$-th row of $A_i$ lies in column $c_i(p)$. The product $g_i\cdot A_i$ on the left side of our equality has as its $q$-th column either the $c_i^{-1}(q)$-th column of $g_i$ (if $q$ lies in the image of $c_i$), or is identically zero otherwise. Conversely, for $p \leq k_i$ the matrix $A_i \cdot g_{i-1}$ on the right side has as its $p$-th row the $c(p)$-th row of $g_{i-1}$, and its rows corresponding to $p > k_i$ are identically zero. 
		
		Therefore, requiring these two products to be equal amounts to imposing three types of constraints on the entries of $g_{i-1}$ and $g_i$:
		\begin{enumerate}
		    \item $g_{i}(p,q) = 0$  whenever $p > k_i \geq q$.
			\item $g_{i-1}(p,q) = 0$ whenever $p \in \text{Img}(c_i)$ and $q \notin \text{Img}(c_i)$. 
			\item $g_{i-1}(c_i(p),c_i(q)) = g_i(p,q)$ whenever both $p$ and $q$ are $\leq k_i$. 
		\end{enumerate}
		Recalling that $(V_\bullet,f_\bullet)$ is the persistence module represented by $A_\bullet$, we have a bijection 
		\[
		\left[\begin{matrix}
		\text{intervals $[i,j]$ in the} \\
		\text{barcode of $(V_\bullet,f_\bullet)$}
		\end{matrix}\right]
		\stackrel{\simeq}{\longleftrightarrow}
		\left[\begin{matrix}
		\text{sequences $\set{p_k \mid i \leq k \leq j}$ with} \\
		\text{$c_k(p_k) = p_{k-1}$ for $i+1 \leq k \leq j$}
		\end{matrix}\right]		
		\]
		Let $[i_1,j_1]$ and $[i_2,j_2]$ be two intervals in the barcode decomposition of $(V_\bullet,f_\bullet)$, and denote their corresponding sequences by $\set{p_\bullet}$ and $\set{q_\bullet}$. It follows from constraint (3) above that $g_k(p_k,q_k)$ remains constant whenever $k$ ranges over the indices in $[i,j] := [i_1,j_1] \cap [i_2,j_2]$. In other words, we have
		\begin{align}\label{eq:gconst}
		g_k(p_k,q_k) = g_{k'}(p_{k'},q_{k'}) \text{ for all } k,k' \in [i,j].
		\end{align} The following observation is crucial.
		
		\noindent {\bf Claim:} The entry $g_k(p_k,q_k)$ is zero for all $k \in [i,j]$ whenever $[i_1,j_1] \not\preceq [i_2,j_2]$. 
		
		\noindent To prove this claim, note that if $i_2 < i_1$ holds then $p_{i_1} > k_{i_1} \geq q_{i_1}$, so $g_{i_1}(p_{i_1},q_{i_1}) = 0$ by constraint (1) above. Thus the claim extends to all $k$ in $[i,j]$ by \eqref{eq:gconst}. Similarly, if $j_2 < j_1$ then $p_{j_2} \in \text{\rm Img}(c_{j_2})$ but $q_{j_2} \not\in \text{\rm Img}(c_{j_2})$, whence $g_{j_2}(p_{j_2},q_{j_2}) = 0$ by constraint (2). Once again, this extends to all $k \in [i,j]$ by \eqref{eq:gconst}, and so the claim is proved.
		
		Returning to the main argument, for each  pair 
		$[i_1,j_1] \preceq [i_2,j_2]$ of intervals in the barcode of $(V_\bullet,f_\bullet)$, we may select some $k \in [i_1,j_1] \cap [i_2,j_2]$. We denote by $g_{[i_1,j_1]}^{[i_2,j_2]}$ the submatrix of $g_k$ spanned by all entries $g_k(p_k,q_k)$ for which  $\set{p_\bullet}$ and $\set{q_\bullet}$ are sequences corresponding to intervals of type $[i_1,j_1]$ and $[i_2,j_2]$ respectively. Thus, $g_{[i_1,j_1]}^{[i_2,j_2]}$ has exactly $d_{i_1j_1}$ rows and $d_{i_2j_2}$ columns; and from \eqref{eq:gconst} we know that it forms a submatrix of $g_k$ for all $k$ in $[i_1,j_1] \cap [i_2,j_2]$. It follows from our claim that each $g_k$ is block upper-triangular:
		\begin{align}\label{eq:gkmatrix}
		\renewcommand{\arraystretch}{2}
	g_{k}=
	\left[
	\begin{array}{cccccccccc}
	g_{[0,k]}^{[0,k]} & g_{[0,k]}^{[0,k+1]} & \cdots  & \cdots & g_{[0,k]}^{[0,\ell]} & 0 & \cdots & \cdots& \cdots & 0\\
	0 & g_{[0,k+1]}^{[0,k+1]} & \cdots & \cdots & g_{[0,k+1]}^{[0,\ell]} & 0& \cdots & \cdots& \cdots & 0\\
	0 & 0& \ddots & \vdots & \vdots & \vdots & \vdots & \vdots & \vdots & \vdots\\
	\vdots & \vdots & 0& \cdots& \cdots & \cdots & \cdots & 0 &g_{[k,\ell-1]}^{ [k,\ell-1]} & g_{[k,\ell-1]}^{[k,\ell]}\\
	0 & 0&  \cdots & \cdots & \cdots & \cdots & \cdots& 0&  0 & g_{[k,\ell]}^{[k,\ell]} \\
	\end{array}
	\right]
	\end{align}
	The fact that $g_k$ must be invertible forces the diagonal blocks to be invertible, while the off-diagonal blocks remain entirely unconstrained. The map
    \[
    \Stab(A_\bullet) \to \prod_{[i,j]}\GL(d_{ij};\mathbb{F})\times \hspace{-.22in} \prod_{[i_1,j_1] \precneq [i_2,j_2] } \hspace{-.2in} \text{\rm Mat}(d_{i_{1}j_{1} }\times d_{i_{2}j_{2}};\mathbb{F}),
    \]
    which sends each $g$ to this distinguished collection of invertible $g_{[i,j]}^{[i,j]}$ and arbitrary $g_{[i_1,j_1]}^{[i_2,j_2]}$ furnishes the desired bijection.
	\end{proof}
	Using the block upper triangular form of the matrices $g_k$ described in the argument above, we may immediately obtain the group structure of $\Stab(A_\bullet)$.
	
	\begin{cor}\label{cor:stabmult}
	For $g,h \in \Stab(A_\bullet)$, we have
	\[
	(gh)_{[i_1,j_1]}^{[i_2,j_2]}= \sum_{[a,b]} g_{[i_1,j_1]}^{[a,b]} h_{[a,b]}^{[i_2,j_2]},
	\]
	with the sum being indexed over intervals that satisfy $[i_1,j_1] \preceq [a,b] \preceq [i_2,j_2]$.
	\end{cor}
	
	 If $\mathbb{F}$ is the field of real or complex numbers, as a subgroup of $\prod_{i=0}^{\ell}\GL_{n_{i}}(\mathbb{F})$,  the group $\Stab(A_\bullet)$ is a Lie group. Theorem \ref{theorem:stabbij} immediately allows us to obtain its dimension.
	\begin{cor} If $\mathbb{F}$ is the field of real or complex numbers,
		$\Stab(A_\bullet)$ is a Lie group of dimension \[\sum_{[i_1,j_1] \preceq [i_2,j_2] }d_{i_{1}j_{1}}d_{i_{2}j_{2}}\]
	\end{cor}
	As stated at the start of this section, our task here is to determine the space of all possible ordered barcode bases for a given persistence module $(V_\bullet, f_\bullet)$.
	
	We denote by $\mathbb{B}=\set{ \B=(B_i)_{0 \leq i \leq \ell} \; \middle| \; B_i \subset V_i \; \text{is an ordered basis}}$ the set of all possible ordered bases of $(V_\bullet, f_\bullet)$. Having fixed an initial basis $\B \in \mathbb{B}$, we know the group $G$ from \eqref{eq:G} acts freely and transitively on the set $\mathbb{B}$, so that any element of $\mathbb{B}$ may be expressed as $g \B$ for some unique $g \in G$. Thus, once we have fixed an initial basis $\B$, the set $\mathbb{B}$ may be identified with $G$. Then as subset of $\mathbb{B}$, the set of all possible ordered barcode bases of $(V_\bullet, f_\bullet)$ can be identified as a subset of the group $G$. 
	
	Recall,  $X$ is the set of all possible matrix-sequences  as defined in \eqref{eq:X}. For each possible basis $\B \in \mathbb{B}$, we define $A(\B)_\bullet \in X$ to be the matrix representation of the linear maps $f_\bullet$ in the chosen basis $\B$. This assignment prescribes the {\em matrix representation map} 
	\[
	A()_\bullet: \mathbb{B} \mapsto X,
	\] and finding all possible ordered barcode bases for $(V_\bullet, f_\bullet)$ amounts to determining all bases $\B \in \mathbb{B}$ for which $A(\B)_\bullet$ is as in \eqref{eq:obform}. Furthermore, this map is equivariant in the sense that $A(g\B)_\bullet=(gA(\B))_\bullet$ for each $g \in G$, where $g$ acts via the basis action defined in \eqref{eq:ga}.

    The following result makes the link between the stabiliser of $A_\bullet$ and the set of ordered barcode bases of a persistence module $(V_\bullet, f_\bullet)$.
        \begin{prop} \label{prop:obbase} Given a persistence module $(V_\bullet, f_\bullet)$ together with an ordered barcode basis $\B$ with matrix representation $A_\bullet$, the set of all ordered barcode bases is given by the orbit $\Stab(A_\bullet) \B$.
    \end{prop}
	\begin{proof}
	Let $\B'$ be another ordered barcode basis. As seen above, two ordered barcode bases have the same matrix representations \eqref{eq:obform}, so that $A(\B')_\bullet=A_\bullet$. As previously stated, $G$ acts freely and transitively on $\mathbb{B}$ so that there exists a unique $g \in G$ for which $\B'=g\B$.
    We then have $A_\bullet=A(\B)_\bullet=A( g \B')_\bullet=gA(\B')_\bullet=gA_\bullet$, so that $g \in \Stab(A_\bullet)$ which implies $\B' \in  \Stab(A_\bullet) \B$.
	\end{proof}
	
	As such, we may identify the set of all ordered barcode bases of a persistence module $(V_\bullet, f_\bullet)$ with $\Stab(A_\bullet)$, which was fully characterised in Theorem \ref{theorem:stabbij}.

\begin{rem}
The automorphism group $\text{Aut}_Q(M)$ of a representation $M$ of a general quiver $Q$ has been described, for instance in \cite[Section 2.2]{brion}. It is known that $\text{Aut}_Q(M)$ is a semi-direct product of the form 
\[
U \rtimes \prod_{i=1}^{r} \GL(m_i,\F).
\] Here $M=\bigoplus_{i=1}^{r}M_i^{m_i}$ is a decomposition of $M$ into indecomposable summands $M_i$, while $U$ is unipotent normal subgroup of $\text{Aut}_Q(M)$ (see \cite[Prop 2.2.1]{brion}). Viewed from this context, the main content of Theorem \ref{theorem:stabbij} is an explicit description of $U$ in the special case where $Q$ is a type-A quiver. In particular, $U$ is generated by matrices which have the form \eqref{eq:gkmatrix}, but with identity blocks along the diagonal. This explicit description of $U$ in the type-A case plays a crucial role in subsequent results which appear in this paper.
\end{rem}	

\section{Simplifying Maps of Persistence Modules}\label{sec:maps}
We now shift our interest to maps of persistence modules $\phi_\bullet : (V_\bullet, f_\bullet) \mapsto (W_\bullet,h_\bullet)$, where $V_\bullet$ and $W_\bullet$ are persistence modules of length $\ell+1$. We recall that each such $\phi_\bullet$ is a collection of linear maps $\phi_i : V_{i} \mapsto W_{i}$ satisfying $\phi_i \circ f_i=h_i \circ \phi_{i-1}$. In other words, the following diagram commutes:
\[
\xymatrixcolsep{.5in}
\xymatrixrowsep{.4in}
\xymatrix{
V_0 \ar@{->}[r]^{f_1} \ar@{->}[d]_{\phi_0} & V_1 \ar@{->}[r]^{f_2} \ar@{->}[d]_{\phi_1} & \cdots \ar@{->}[r]^{f_{\ell-1}}  & V_{\ell-1} \ar@{->}[r]^{f_\ell} \ar@{->}[d]_{\phi_{\ell-1}}& V_\ell \ar@{->}[d]_{\phi_{\ell}} \\
W_0 \ar@{->}[r]_{h_1} & W_1 \ar@{->}[r]_{h_2} & \cdots \ar@{->}[r]_{h_{\ell-1}} & W_{\ell-1} \ar@{->}[r]_{h_\ell} & W_\ell
}
\]
Our objective here is to show that if neither $V_\bullet$ nor $W_\bullet$ admits a pair of strictly nested bars in its barcode, then there exist barcode bases of $V_\bullet$ and $W_\bullet$ in which $\phi$ induces a {\em partial matching} of the bars. To make this precise, we first describe the nestedness condition.

\begin{definition}\label{def:nested} A bar $[i_2,j_2]$ is strictly nested in a bar $[i_1,j_1] $, denoted $[i_2,j_2] \subset [i_1,j_1]$ if $i_1 < i_2 \leq j_2 < j_1$. This is best represented as 

  
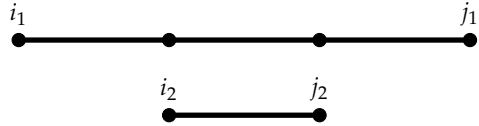
\begin{tikzpicture}[line cap=round,line join=round,>=triangle 45,x=1cm,y=1cm]
		\clip(-10.853170731707314,3.75) rectangle (8.551707317073165,5.56487804878048);
		\draw [line width=2pt] (-4,5)-- (-2,5);
		\draw [line width=2pt] (-2,5)-- (0,5);
		\draw [line width=2pt] (-6,5)-- (-4,5);
		\draw [line width=2pt] (-4,4)-- (-2,4);
		
		\begin{scriptsize}
		\draw [fill=black] (-6,5) circle (2.5pt);
		\draw [color=black] (-6,5.35) node {$i_1$};
		\draw [fill=black] (-4,5) circle (2.5pt);
		\draw [fill=black] (-2,5) circle (2.5pt);
		\draw [fill=black] (0,5) circle (2.5pt);
		\draw [color=black] (0,5.35) node {$j_1$};
		\draw [fill=black] (-4,4) circle (2.5pt);
		\draw [color=black] (-4,4.35) node {$i_2$};
		\draw [fill=black] (-2,4) circle (2.5pt);
		\draw [color=black] (-2,4.35) node {$j_2$};
		\end{scriptsize}
		\end{tikzpicture}
\end{definition}
\noindent Note that two intersecting bars are strictly nested if and only if they are not related by $\preceq$.

Before stating the main result, we remark that the data of our map $\phi_\bullet$ can be interpreted as a representation of the {\em rectangle} quiver of length $\ell+1$:
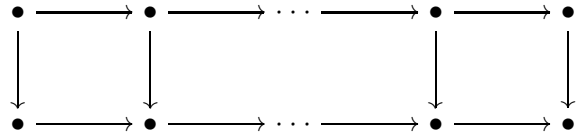
\[
\xymatrixcolsep{.5in}
\xymatrixrowsep{.4in}
\xymatrix{
\bullet \ar@{->}[r]^{} \ar@{->}[d]_{} & \bullet \ar@{->}[r]^{} \ar@{->}[d]_{} & \cdots \ar@{->}[r]^{}  & \bullet \ar@{->}[r]^{} \ar@{->}[d]_{}& \bullet \ar@{->}[d]_{} \\
\bullet \ar@{->}[r] & \bullet \ar@{->}[r]_{} & \cdots \ar@{->}[r]_{} & \bullet \ar@{->}[r]_{} & \bullet
}
\]
Such representations have been called \textbf{ladder} persistence modules in the literature  \cite{escolar2015persistence}. When treating $\phi$ as a ladder persistence module, we will denote it $(V_\bullet, W_\bullet, \phi_\bullet)$. We may therefore seek to decompose $\phi_\bullet$ into a direct sum of indecomposable ladder persistence modules.  

\begin{definition} Three families of ladder persistence modules are defined below:
\begin{enumerate}
    \item Given intervals $[i_1,j_1] \preceq [i_2,j_2]$, denote by $\textbf{R}_{[i_1,j_1]_\bullet}^{[i_2,j_2]}$ the ladder persistence module where $V_\bullet$ is the interval module $\Int[i_2,j_2]_\bullet$ while $W_\bullet$ is the interval module $\Int[i_1,j_1]_\bullet$; all vertical maps are $1'$s whenever possible and $0$ otherwise:
    \[
    \xymatrixcolsep{.3in}
    \xymatrixrowsep{.4in}
    \xymatrix{
    & & \mathbb{F} \ar@{->}[r] \ar@{->}[d]  & \mathbb{F} \ar@{->}[r] \ar@{->}[d]  & \cdots \ar@{->}[r] & \mathbb{F}  \ar@{->}[r] \ar@{->}[d] & \cdots \ar@{->}[r]& \mathbb{F} \\
    \mathbb{F}  \ar@{->}[r] & \cdots \ar@{->}[r]& \mathbb{F} \ar@{->}[r]  & \mathbb{F} \ar@{->}[r]  & \cdots \ar@{->}[r]& \mathbb{F}   
    }
\]
  \item Given an interval $[i,j]$, let $\Int^+[i,j]_\bullet$ denote the ladder persistence module for which $V_\bullet$ is $\Int[i,j]_\bullet$ and $W_\bullet$ is $0$, with all vertical maps necessarily being $0$:
  \[
\xymatrixcolsep{.3in}
\xymatrixrowsep{.4in}
\xymatrix{
 & & \mathbb{F} \ar@{->}[r] \ar@{->}[d]  & \mathbb{F} \ar@{->}[r] \ar@{->}[d] & \cdots \ar@{->}[r]  & \mathbb{F} \ar@{->}[d]\\
0  \ar@{->}[r] & \cdots \ar@{->}[r]& 0 \ar@{->}[r]  & 0 \ar@{->}[r]  & \cdots \ar@{->}[r]& 0 \ar@{->}[r] & \cdots \ar@{->}[r]& 0    
}
\]
\item And finally, given an interval $[i,j]$, let $\Int^-[i,j]_\bullet$ be the ladder persistence module for which $V_\bullet$ is trivial while $W_\bullet$ is $\Int[i,j]_\bullet$, so once again all vertical maps are $0$:
\[
\xymatrixcolsep{.3in}
\xymatrixrowsep{.4in}
\xymatrix{
0  \ar@{->}[r] & \cdots \ar@{->}[r]& 0 \ar@{->}[r] \ar@{->}[d] & 0 \ar@{->}[r] \ar@{->}[d]  & \cdots \ar@{->}[r] & 0 \ar@{->}[r] \ar@{->}[d]& \cdots \ar@{->}[r]& 0 \\
& & \mathbb{F} \ar@{->}[r]   & \mathbb{F} \ar@{->}[r]  & \cdots \ar@{->}[r]  &  \mathbb{F}
}
\]
\end{enumerate}
\end{definition}

It is readily seen that these three families of ladder persistence modules are mutually non-isomorphic and indecomposable. Therefore, by the Krull-Schmidt theorem, if $\phi_\bullet$ were to decompose as a direct sum of modules sourced from these three families, then such a decomposition would be unique. From such a decomposition, we can obtain the desired partial matching of the source and target bars: the presence of each $\textbf{R}_{[i_1,j_1]_\bullet}^{[i_2,j_2]}$ summand matches a bar $[i_2,j_2]$ of $V_\bullet$ to a bar $[i_1,j_1]$ of $W_\bullet$, whilst the existence of $\Int^+[i,j]_\bullet$ (or $\Int^-[i,j]_\bullet$) summands reveals bars $[i,j]$ in $V_\bullet$ (or $W_\bullet)$ that are matched to $0$. With this in mind, here is the main result of this section.

 \begin{theorem} {\label{theorem:nonest}}
Let $(V_\bullet, W_\bullet, \phi_\bullet)$ be a ladder persistence module of length $\ell+1$ where neither $V_\bullet$ nor $W_\bullet$ admit a pair of strictly nested bars. Then there are integers $r_{[i_1,j_1]} ^{[i_2,j_2]}$ and $d_{ij}^\pm \in \mathbb{N}$ for which:
      \[(V_\bullet, W_\bullet, \phi_\bullet) \hspace{-.1em} \simeq \hspace{-.3em} \bigoplus_{[i_1,j_1] \hspace{-.3em} \preceq [i_2,j_2]}  \left(\textbf{R}_{[i_1,j_1]_\bullet}^{[i_2,j_2]}\right)^{r_{[i_1,j_1]}^{[i_2,j_2]}} \oplus \bigoplus_{i \leq j} \left(\Int^{+}[i,j]_\bullet\right)^{d_{ij}^{+}} \oplus \bigoplus_{i \leq j} \left(\Int^{-}[i,j]_\bullet\right)^{d_{ij}^{-}}
      \]
  (In Example \ref{example: bars counter example} below we show that the assumption precluding nested bars is necessary.)
  \end{theorem}
  
\begin{proof}
The argument proceeds along three basic steps.

\textbf{Step 1: Representing $\phi$}. 
Consider ordered barcode bases 
\[ \B_{V} := \set{B_{V,i} \subset V_i \mid 0 \leq i \leq \ell} \text{ and } \B_{W} := \set{B_{W,i} \subset W_i \mid 0 \leq i \leq \ell}\]
of $V_\bullet$ and $W_\bullet$. Let $d_{i,j}^{V}$ and $d_{i,j}^{W}$ be the multiplicity of $[i,j]$ bars in the barcodes of $V_\bullet$ and $W_\bullet$. We denote by $b_\bullet$ the matrix representations of the maps $\phi_\bullet$ in these chosen bases.
As in the proof Theorem \ref{theorem:stabbij}, let $[i_1,j_1]$ and $[i_2,j_2]$ be two intervals in the barcode decomposition of $(V_\bullet,f_\bullet)$, and denote their corresponding sequences by $\set{p_\bullet}$ and $\set{q_\bullet}$. Using the commuting relations $\phi_k \circ f_k=h_k \circ \phi_{k-1}$,  we have $\phi_k(p_k,q_k)=\phi_{k'}(p_{k'},q_{k'}) $, for all $k, k' \in [i,j]=[i_1,j_1] \cap [i_2,j_2]$, with this coefficient being zero unless $[i_1,j_1] \preceq [i_2,j_2]$. Assuming this order relation holds, define $X_{[i_1,j_1]}^{[i_2,j_2]}$, to be the submatrix of $b_{i}$ obtained by taking the $d_{i_2,j_2}^{V}$ columns corresponding to basis vectors part of an $[i_1,j_1]$ bar of $V$, and the $d_{i_1,j_1}^{W}$ rows corresponding to basis to basis vectors part of an $[i_2,j_2]$ bar of $W$. From the above observation, $X_{[i_1,j_1]}^{[i_2,j_2]}$ is a submatrix of $b_{i}, b_{i+1}, \dots b_{j}$, and is of dimension $d_{i_1,j_1}^{W} \times d_{i_2,j_2}^{V}$. Thus, the matrices $b_\bullet$ are completely determined by the matrices $X_{[i_1,j_1]}^{[i_2,j_2]}$ and may therefore be represented as a single block matrix
\[
	\left[
	\begin{array}{ccccccc}
	X_{[0,0]}^{ [0,0]} & X_{[0,0]}^{ [0,1]} & \dots & X_{[0,0]}^{[0, \ell]} &0  & \dots&0\\
	0 & X_{[0,1]}^{[0,1]} & \dots & \dots & X_{[0,1]}^{[1,\ell]} & \dots  & 0 \\
	\vdots& \vdots & \ddots& & & &  \\
	& & &  \ddots \\
	& & & & \ddots \\
	& & & &  & X_{[\ell, \ell-1]}^{ [\ell, \ell-1]} & X_{[\ell, \ell-1]}^{ [\ell, \ell]}\\
	0& 0 & &  &  & 0 & X_{[\ell, \ell]}^{ [\ell, \ell]} \\
	\end{array}
	\right] 
	\]
	This is a matrix of size $(\sum d_{i,j}^{W}) \times (\sum d_{i,j}^{V})$.
	
\textbf{Step 2: Admissible Operations}. Given two ordered barcode bases $\B_{V}$ and $\B_{w}$ we may define a block matrix as above, which we denote $b(\B_V,\B_W)$. By Proposition \ref{prop:obbase}, the set of ordered all barcode bases of $(V_\bullet, f_\bullet)$ coincides precisely with the orbit $\Stab(A(\B_1)_\bullet) \B_1$, where $\B_1$ is an ordered barcode basis. Then given $(h,k) \in \Stab(A(\B_V)_\bullet) \times \Stab(A(\B_W)_\bullet)$, we may consider $b(h\B_V,k \B_W)$. As seen in the proof of Theorem \ref{theorem:stabbij}, we are able to completely characterise an element $h \in \Stab(A(\B_V)_\bullet)$ in terms of the submatrices $h_{[i_1,j_1]}^{ [i_2,j_2]}$, so that $h$ may be represented as a single block matrix 

\[
	\left[
	\begin{array}{ccccccc}
	h_{[0,0]}^{ [0,0]} & h_{[0,0]}^{ [0,1]} & \dots & h_{[0,0]}^{[0, \ell]} &0  & \dots&0\\
	0 & h_{[0,1]}^{[0,1]} & \dots & \dots & h_{[0,1]}^{[1,\ell]} & \dots  & 0 \\
	\vdots& \vdots & \ddots& & & &  \\
	& & &  \ddots \\
	& & & &  & h_{[\ell, \ell-1]}^{ [\ell, \ell-1]} & h_{[\ell, \ell-1]}^{ [\ell, \ell]}\\
	0& 0 & &  &  & 0 & h_{[\ell, \ell]}^{ [\ell, \ell]} \\
	\end{array}
	\right] 
	\]
of size $(\sum d_{i,j}^{V}) \times (\sum d_{i,j}^{V})$. The same is true for $k \in \Stab(A(\B_W)_\bullet)$, whence $b(h \B_V,k \B_W)$ equals the (matrix) product $k \cdot b(\B_V,\B_W) \cdot h^{-1}$. Thus, we are only allowed to perform the following legal operations on $b(\B_V,\B_W)$:
\begin{enumerate}
    \item Using invertible block diagonal elements $h_{[i_1,j_1]}^{[i_1,j_1]}$ and $k_{[i_1,j_1]}^{[i_1,j_1]}$, we may perform any operations between column and rows corresponding to $[i_1,j_1]$ bars in our block matrix.
    
    \item Using the block matrices $h_{[i_1,j_1]}^{ [i_2,j_2]}$, we see we may modify columns corresponding to $[i_2,j_2]$ bars using columns corresponding to $[i_1,j_1]$ in our block matrix, whenever $[i_1,j_1] \preceq [i_2,j_2]$.
    
    \item Using the block matrices $k_{[i_1,j_1]}^{ [i_2,j_2]}$, we see we may modify rows corresponding to $[i_1,j_1]$ bars using columns corresponding to $[i_2,j_2]$ in our block matrix, whenever $[i_1,j_1] \preceq [i_2,j_2]$.
\end{enumerate}

\textbf{Step 3: Matrix Reduction}. We wish to find ordered barcode bases for which the corresponding matrix $b(\B_V,\B_W)$ admits at most one non-zero term $1$ in each row and column. From this form, the desired decomposition can be easily extracted: every $1$ in a row corresponding to an $[i_1,j_1]$ bar and column corresponding to an $[i_2,j_2]$ bar (with $[i_1,j_1] \preceq [i_2,j_2]$) corresponds to an $\textbf{R}_{[i_1,j_1]_\bullet}^{[i_2,j_2]}$ summand. Similarly, zero rows and columns then yield $\Int^{-}[i,j]_\bullet$ and $\Int^{+}[i,j]_\bullet$ summands respectively.

	Let $\B_V$, $\B_W$ be ordered barcode bases of $V_\bullet$, $W_\bullet$ so that 
	\[
	b(\B_V,\B_W)= \left[
	\begin{array}{ccccccc}
	X_{[0,0]}^{ [0,0]} & X_{[0,0]}^{ [0,1]} & \dots & X_{[0,0]}^{[0, \ell]} &0  & \dots&0\\
	0 & X_{[0,1]}^{[0,1]} & \dots & \dots & X_{[0,1]}^{[1,\ell]} & \dots  & 0 \\
	\vdots& \vdots & \ddots& & & &  \\
	& & &  \ddots \\
	& & & & \ddots \\
	& & & &  & X_{[\ell, \ell-1]}^{ [\ell, \ell-1]} & X_{[\ell, \ell-1]}^{ [\ell, \ell]}\\
	0& 0 & &  &  & 0 & X_{[\ell, \ell]}^{ [\ell, \ell]} \\
	\end{array}
	\right] 
	\]
	We seek stabiliser elements $(h,k)$ so that $B(h\B_V,k\B_W)=kB(\B_V,\B_W)h^{-1}$ is in barcode form. To this end, we perform basis changes using legal operations of type (1), (2) and (3). We process the column-blocks of this matrix from left to right, in each case starting from the diagonal block and working our way upwards. We will denote each treated matrix that has been put in adequate form by $P_{[i_1,j_1]}^{[i_2,j_2]}$. \\
	
	That is, we start with $X_{[0,0]}^{[0,0]}$, putting it in Smith normal form using basis changes $h_{[0,0]}^{[0,0]}$ and $k_{[0,0]}^{ [0,0]}$. Now assume we wish to treat $X_{[i_1,j_1]}^{[i_2,j_2]}$, where all matrices below it and to its left have been treated. That is, we have 
	\[\begin{blockarray}{cccccccc}
	\begin{block}{[cccccccc]}
	&  & & & && \\
	0& \cdots  & & P_{[i_1,j_1]}^{[i_1,j_1]} & \cdots & X_{[i_1,j_1]}^{[i_2,j_2]} & \\
	&  & && & \vdots & \\
	&  & & && P_{[i_2,\ell]}^{[i_2,j_2]} & \\
	& & & && 0 & \\
	& & & && \vdots & \\
	& & & && 0 & \\
	\end{block}
	\end{blockarray}
	\]
	Given non-intersecting bars $[a,b] \trianglelefteq [c,d]$, we have either $[a,b] \preceq [c,d]$ or $[a,b] \subset [c,d]$. Then by hypothesis, given $P_{[i_1,j_1]}^{[a,b]}$ to the left of $X_{[i_1,j_1]}^{[i_2,j_2]}$, we have $[a,b] \preceq [i_2,j_2]$.  So we may zero out rows of $X_{[i_1,j_1]}^{[i_2,j_2]}$ in which $P_{[i_1,j_1]}^{[a,b]}$ has 1's using operations of type (2). Similarly, given $P_{[a,b]}^{[i_2,j_2]}$ below $X_{[i_1,j_1]}^{[i_2,j_2]}$, we may zero out corresponding columns using operations of type (3). The non-zero columns of the resulting matrix $\tilde{X}_{[i_1,j_1]}^{ [i_2,j_2]}$ have 0's below them, and non-zero rows have 0's to their left.
	
	Then using basis changes $h_{[i_2,j_2]}^{[i_2,j_2]}$ and $k_{[i_1,j_1]}^{[i_1,j_1]}$, we may put $A$ in Smith normal form, without adding non-zero terms in any rows below and column to its left, preserving the desired structure.
	\end{proof}
	
\begin{ex} \label{example: bars counter example} We illustrate the difficulties imposed by strictly nested bars in the context of Theorem \ref{theorem:nonest}. 
Consider the map $\phi_\bullet : V_\bullet \mapsto W_\bullet$ where $V_\bullet$ has barcode 

	
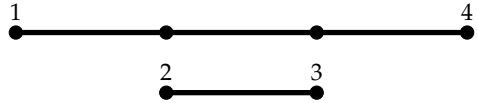
\begin{tikzpicture}[line cap=round,line join=round,>=triangle 45,x=1cm,y=.8cm]
		\clip(-8.853170731707314,3) rectangle (8.551707317073165,5.56487804878048);
		\draw [line width=2pt] (-4,5)-- (-2,5);
		\draw [line width=2pt] (-2,5)-- (0,5);
		\draw [line width=2pt] (0,5)-- (2,5);
		\draw [line width=2pt] (-2,4)-- (0,4);
		
		\begin{scriptsize}
		\draw [fill=black] (-4,5) circle (2.5pt);
		\draw [color=black] (-4,5.35) node {1};
		\draw [fill=black] (-2,5) circle (2.5pt);
		\draw [fill=black] (0,5) circle (2.5pt);
		\draw [fill=black] (2,5) circle (2.5pt);
		\draw [color=black] (2,5.35) node {4};
		\draw [fill=black] (-2,4) circle (2.5pt);
		\draw [color=black] (-2,4.35) node {2};
		\draw [fill=black] (0,4) circle (2.5pt);
		\draw [color=black] (0,4.35) node {3};
		\end{scriptsize}
		\end{tikzpicture}
		
\noindent and $W_\bullet$ has barcode

\begin{tikzpicture}[line cap=round,line join=round,>=triangle 45,x=1cm,y=.8cm]
		\clip(-8.853170731707314,3) rectangle (8.551707317073165,5.56487804878048);
		\draw [line width=2pt] (-4,5)-- (-2,5);
		\draw [line width=2pt] (-2,5)-- (0,5);
		\draw [line width=2pt] (-6,5)-- (-4,5);

		\begin{scriptsize}
		\draw [fill=black] (-6,5) circle (2.5pt);
		\draw [color=black] (-6,5.35) node {0};
		\draw [fill=black] (-4,5) circle (2.5pt);
		\draw [fill=black] (-2,5) circle (2.5pt);
		\draw [fill=black] (0,5) circle (2.5pt);
		\draw [color=black] (0,5.35) node {3};
		\end{scriptsize}
		\end{tikzpicture}
		with associated ordered barcode bases $\B_V$, $\B_W$, and $\phi_\bullet$ is given by the block matrix representation 
	 \[b(\B_V,\B_W) = ~
	\begin{blockarray}{ccc}
	[1,4]&[2,3] &\\
	\begin{block}{[cc]c}
	1 & 1& [0,3] \\
	\end{block}
	\end{blockarray}
	\]
	Since $[2,3] \subset [1,4]$, stabiliser changes of basis for $\B_V$ are invertible diagonal matrices $\left[\begin{smallmatrix} a & 0 \\ 0 & b \end{smallmatrix}\right]$, and stabiliser changes of basis for $\B_W$ are invertible matrices $\left[\begin{smallmatrix} c \end{smallmatrix} \right]$. As such, $(V_\bullet, W_\bullet, \phi_\bullet)$
	which will never be expressible as a direct sum of modules as in Theorem \ref{theorem:nonest}, since there is no change of basis which will allow us to transform this matrix into either $\left[\begin{smallmatrix} 0 & 1 \end{smallmatrix} \right]$ or $\left[\begin{smallmatrix} 1 & 0 \end{smallmatrix} \right]$. The same is true if $V_\bullet$ has barcode 
	
	\begin{tikzpicture}[line cap=round,line join=round,>=triangle 45,x=1cm,y=.8cm]
		\clip(-8.853170731707314,3) rectangle (8.551707317073165,5.56487804878048);
		\draw [line width=2pt] (-4,5)-- (-2,5);
		\draw [line width=2pt] (-2,5)-- (0,5);
		\draw [line width=2pt] (0,5)-- (2,5);

		\begin{scriptsize}
		\draw [fill=black] (-4,5) circle (2.5pt);
		\draw [color=black] (-4,5.35) node {1};
		\draw [fill=black] (-2,5) circle (2.5pt);
		\draw [fill=black] (0,5) circle (2.5pt);
		\draw [fill=black] (2,5) circle (2.5pt);
		\draw [color=black] (2,5.35) node {4};
		\end{scriptsize}
		\end{tikzpicture}
		
\noindent and $W_\bullet$ has barcode

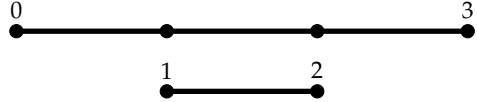
\begin{tikzpicture}[line cap=round,line join=round,>=triangle 45,x=1cm,y=.8cm]
		\clip(-8.853170731707314,3) rectangle (8.551707317073165,5.56487804878048);
		\draw [line width=2pt] (-4,5)-- (-2,5);
		\draw [line width=2pt] (-2,5)-- (0,5);
		\draw [line width=2pt] (-6,5)-- (-4,5);
		\draw [line width=2pt] (-4,4)-- (-2,4);
		
		\begin{scriptsize}
		\draw [fill=black] (-6,5) circle (2.5pt);
		\draw [color=black] (-6,5.35) node {0};
		\draw [fill=black] (-4,5) circle (2.5pt);
		\draw [fill=black] (-2,5) circle (2.5pt);
		\draw [fill=black] (0,5) circle (2.5pt);
		\draw [color=black] (0,5.35) node {3};
		\draw [fill=black] (-4,4) circle (2.5pt);
		\draw [color=black] (-4,4.35) node {1};
		\draw [fill=black] (-2,4) circle (2.5pt);
		\draw [color=black] (-2,4.35) node {2};
		\end{scriptsize}
		\end{tikzpicture}
			with associated ordered barcode bases $\B_V$, $\B_W$, and $\phi_\bullet$ is given by the block matrix representation
	\[
	\begin{blockarray}{cc}
	[1,4] &\\
	\begin{block}{[c]c}
	1 &  [0,3] \\
	1 & [1,2]\\
	\end{block}
	\end{blockarray}
	\]

	Thus, if the nested condition is violated on either $V_\bullet$ or $W_\bullet$, a decomposition as in Theorem \ref{theorem:nonest} is not always possible.
\end{ex} 

\begin{rem} \label{example: potential applications}
The non-nestedness hypothesis on the bars  of $V_\bullet$ and $W_\bullet$ from Theorem \ref{theorem:nonest} is quite restrictive. There are, however, several scenarios of interest where it is satisfied: 
\begin{enumerate}
    \item The 0-th persistent homology of a point cloud satisfies the hypothesis because all bars have left endpoint $0$.
    \item Similarly, the 1st persistent homology of a filtered graph admits no strictly nested bars because all bars have right endpoint $\infty$.
    \item More generally, the $n$-th persistent homology of an $n$-dimensional filtered complex satisfies the non-nestedness criterion. One may consider, for instance, the {\em Linial-Meshulam} model of random simplicial complexes \cite{linmesh}. A random simplicial complex chosen from this model on a given vertex set consists of every possible simplex of dimension $< n$, with candidate $n$-simplices being included independently with uniform probability $p \in [0,1]$. The $n$-th persistent homology of any filtration of such a random complex satisfies the hypothesis.
\end{enumerate}  Theorem \ref{theorem:nonest} applies in all such cases.
\end{rem}

	\section{Zizag Modules} \label{sec:zigzag}
	
	In this section, we wish to generalise the previous work to representations of any type $A$ quiver, in other words zizag persistence modules. The nomenclature we adpot here is extracted from \cite{zizag}. A zizag module is given by a sequence 
	\[
	\xymatrixcolsep{.5in}
	\xymatrix{
		V_{0} \ar@{<->}[r]^-{p_1} & V_{1} \ar@{<->}[r]^-{p_2} & \cdots  \ar@{<->}[r]^-{p_{\ell-1}} & V_{\ell-1} \ar@{<->}[r]^-{p_\ell} & V_{\ell}
	}
	\]
	where each $\xlongleftrightarrow{p_i}$ is either a forward map $\xlongrightarrow{f_i}$ or a backwards map $\xlongleftarrow{q_i}$. The direction of the arrows define the \textbf{type} $\tau$ of the zizag module, which is the direction of the arrows of the underlying type $A$ quiver. For example,
	\[	\xymatrixcolsep{.5in}
	\xymatrix{
		V_{0} \ar@{->}[r]^-{f_1} & V_{1} \ar@{<-}[r]^-{q_2} & V_2
	}
	\] has type $\tau=fq$. We denote zizag modules as $(V_\bullet, p_\bullet, \tau)$. 
	
   The {\bf interval module} of type $\tau$ corresponding to a pair of non-negative integers $i \leq j$ is the zizag module $\Int_\tau[i,j]_\bullet$ given by
    \[
	\xymatrixcolsep{.17in}
	\xymatrix{
		0 \ar@{<->}[r] & \cdots \ar@{<->}[r]  & 0 \ar@{<->}[r] & \F \ar@{<->}[r] & \cdots \ar@{<->}[r]  & \F \ar@{<->}[r] & 0 \ar@{<->}[r] & \cdots \ar@{<->}[r] & 0,
	}
	\]
	where the contiguous string of $\F$'s spans $\set{i,i+1,\ldots,j-1,j}$, all intermediate $\F\longleftrightarrow\F$ maps are identities in the direction depending on $\tau$, and all other vector spaces are trivial. Zizag modules also decompose into interval modules: this follows from the main result of \cite{gabriel}, and is established more directly in \cite{zizag}. In particular, every zigzag module $(V_\bullet, p_\bullet, \tau)$ is isomorphic to a direct sum 
	\[(V_\bullet, p_\bullet, \tau) \cong \bigoplus_{0 \leq i \leq j \leq \ell} ( \Int _\tau [i,j])^{d_{ij}},
	\] for some uniquely determined $d_{ij} \in \mathbb{N}$. 
	
		\begin{definition} \label{def:robarbase}
	An $m \times n$ matrix $A$ of rank $r$ is in \textbf{reversed barcode form} if there exists a strictly increasing function $c : \set{m-r+1, \ldots, m} \to \set{1, \ldots n}$ so that
	\[
	A_{ij} = \begin{cases}
	                1 & \text{if } j = c(i), \\
	                0 & \text{otherwise.}
	            \end{cases}
	\]
	\end{definition}
	Thus, a matrix is in reversed barcode form whenever its entries lie in $\set{0,1}$, with at most one non-zero term in each row and column, {\em and} satisfies the following additional requirement: the $r$ non-zero terms appear in the last $r$ columns, with increasing row order. (We warn the reader that if a matrix $A$ is in barcode form, then its transpose will not in general be in reversed barcode form; however, the off-diagonal transpose of $A$ will be in reversed barcode form).
	
	\begin{definition}
	A basis family $\B$ for a zizag module is called an {\bf barcode basis} for $(V_\bullet,p_\bullet, \tau)$ if all of the $A_i$ for which $V_{i-1} \xlongrightarrow{f_i} V_{i}$ are in barcode form, and all the $A_i$ for which $V_{i-1} \xlongleftarrow{q_i} V_{i}$ are in reversed barcode form. The corresponding matrices $A_\bullet$ are then said to be in \textbf{zizag barcode form}.
	\end{definition}
	 We note that this notion coincides with Definition \ref{def:barbase} whenever all maps in sight are forward. Barcode bases for zizag modules are the natural bases arising from their decomposition into interval modules.
	
	\subsection{Algorithm for Zizag Modules} 
	
	We wish to generalise our algorithm from Section \ref{sec:barbasis} to treat the case of zizag modules. Instead of presenting concrete algorithms, we will adapt our proof of Proposition \ref{prop:gabcons} in Section \ref{sec:barbasis} to the more general setting. Implementing the algorithm is then achieved in  similar fashion to what was done in Section \ref{sec:barbasis} for classical persistence modules. To this end, we fix a zizag module $(V_\bullet,p_\bullet, \tau)$ expressed as a sequence of matrices $A_\bullet$ with respect to an arbitrary (i.e., not necessarily barcode) basis family $\B$:
	 \begin{align}\label{eq:zig}
	 \xymatrixcolsep{.5in}
	 \xymatrix{
	 	\F^{n_0} \ar@{<->}[r]^-{A_1} & \F^{n_1} \ar@{<->}[r]^-{A_2} & \cdots  \ar@{<->}[r]^-{A_{\ell-1}} & \F^{n_{\ell-1}} \ar@{<->}[r]^-{A_\ell} & \F^{n_\ell}
	 	}
	 \end{align}
	
	The matrices $A_{i}$ are of dimension either $n_{i} \times n_{i-1}$ or $n_{i-1} \times n_{i}$, depending on the type $\tau$ of $V_\bullet$. Analogously to Section \ref{sec:setting}, we may define $X_\tau$ to be the set of all the possible matrix-sequences $A_\bullet$ which can arise in \eqref{eq:zig}. It is a (strict) subset 
	\begin{align*}
	    X_\tau \subset \prod_{i=1}^{\ell}Y_i, \quad \text{where} \quad
	Y_i=
	\begin{cases} \text{Mat}\left(n_i \times n_{i-1};\F\right) \quad \text{if} \quad i-1 \longrightarrow i \\
	\text{Mat}\left(n_{i-1} \times n_i;\F\right) \quad \text{if} \quad i-1 \longleftarrow i 
	\end{cases}
	\end{align*}
	Changes of bases for zigzag modules are obtained by a new action of the group $G$ from \eqref{eq:G} on the set $X_\tau$; the major difference between this action and the one treated in Section \ref{sec:barbasis} is that the (direction of) conjugation now depends on the type $\tau$. Explicitly, if $V_{i-1} \xlongrightarrow{f_i} V_{i}$ points forward, then it gets sent to $g_i \circ f_i \circ g_{i-1}^{-1}$ as before; and conversely, if $V_{i-1} \xlongleftarrow{q_i} V_{i}$ points backwards, then it is sent to $g_{i-1} \circ q_i \circ g_i^{-1}$. For this action, we obtain the following analogue of Lemma \ref{lem:matop}.

		\begin{lemma} \label{lemma:zmatop}
		Assume that the first $\ell-1$ matrices of $\set{A_i \mid 1 \leq i < \ell}$ from \eqref{eq:zig} are in zizag barcode form.
		\begin{itemize}
		\item If $f_\ell: V_{\ell-1} \rightarrow V_\ell$, and the last matrix $A_\ell$ has a pivot in the $(r,q)$ position\footnote{i.e., we have $A_{\ell}(r,q)=1$ while all other entries in the $q$-th column of $A_\ell$ are zero.},  together with a nonzero entry $\alpha := A_\ell(r,p)$ in the same row $r$ but some other column $p > q$, then there exists $g \in  G$ with $g_{\ell}=\text{\rm Id}$ so that $(gA)_\bullet$ equals $A_\bullet$ except $A_\ell$ where the $\alpha$ entry is replaced by zero. 
		
		\medskip
		
		\item If $q_\ell: V_{\ell-1} \leftarrow V_\ell$, and the last matrix $A_\ell$ has a pivot in the $(r,q)$ position,  together with a nonzero entry $\alpha := A_\ell(s,q)$ in the same column $q$ but some other row $ s<r$, then there exists $g \in  G$ with $g_{\ell}=\text{\rm Id}$ so that $(gA)_\bullet$ equals $A_\bullet$ except $A_\ell$ where the $\alpha$ entry is replaced by zero.
		\end{itemize}
			\end{lemma}
 
		\begin{proof}
		The proof is done by induction, in similar fashion to that of Lemma \ref{lem:matop}. The case $\ell=1$ remains trivial. To prove the induction, 4 cases should now be considered. 
		
		\textbf{Case 1} :$V_{\ell-2} \rightarrow V_{\ell-1} \rightarrow V_{\ell}$
		
		This case is precisely that of Lemma \ref{lem:matop}, and the proof remains, the same. 
		
		\textbf{Case 2}:  $V_{\ell-2} \leftarrow V_{\ell-1} \rightarrow V_{\ell}$ 
		
		Performing  $\bC_{p \gets q}(-\alpha)$ on $A_\ell$ induces the same operation on $A_{\ell-1}$. If the $q$-th column of $A_{\ell-1}$ is identically zero, this operation leaves $A_{\ell-1}$ unchanged. Otherwise, since $p>q$ and $A_{\ell-1}$ is in reversed barcode form, if the $q$-th column isn't zero then the $p$-th column is also non-zero. Then the resulting matrix $A_{\ell-1}'$ has the form 
				\[A'_{\ell-1} = ~
	\begin{blockarray}{cccccccc}
	& q & & & & p\\
	\begin{block}{[ccccccc]c}
	0 & 1 & 0 & \cdots &0 & -\alpha & 0 & c \\
	& 0& & & & 0 \\
	& \vdots & & & & \vdots \\
	& 0& &  & & 0\\
	0& 0&0 & \cdots& &1 & 0 & d  \\
	\end{block}
	\end{blockarray}
	\] where we may again use induction. 
	
	\textbf{Case 3}: $V_{\ell-2} \rightarrow V_{\ell-1} \leftarrow V_{\ell}$
	
	We must perform the row operation $\bR_{s \gets r}(-\alpha)$ on $A_\ell$, inducing the same row operation on $A_{\ell-1}$. Again, if the $r$-th row of $A_{\ell-1}$ is identically zero, we are done. Otherwise, since $A_{\ell}$ is in barcode form and $s<r$, if the $r$-th row isn't zero then the $s$-th row isn't zero, and so the resulting matrix $A_{\ell-1}'$ has the form
	$A_{\ell-1}'$ has the form 
			\[A'_{\ell-1} = ~
	\begin{blockarray}{cccccccc}
	& c & & & & d\\
	\begin{block}{[ccccccc]c}
	0 & 1 & 0 & \cdots &0 & -\alpha & 0 & s \\
	& 0& & & & 0 \\
	& \vdots & & & & \vdots \\
	& 0& &  & & 0\\
	0& 0&0 & \cdots& &1 & 0 & r  \\
	\end{block}
	\end{blockarray}
	\] where we may again use induction.
	
	\textbf{Case 4}:  $V_{\ell-2} \leftarrow V_{\ell-1} \leftarrow V_{\ell}$ 
	
	We perform  the row operation $\bR_{s \gets r}(-\alpha)$ on $A_\ell$, inducing the column operation $\bC_{r \gets s}(\alpha)$ on $A_{\ell-1}$. Again, if the $s$-th column of $A_{\ell-1}$ is zero, we are done. Otherwise, since $A_{\ell-1}$ is in barcode form and $s<r$, the $r$-th column is also non-zero so that the resulting matrix $A_{\ell-1}'$ has the form 
		\[A'_{\ell-1} = ~
	\begin{blockarray}{cccccccc}
	& s & & & & r\\
	\begin{block}{[ccccccc]c}
	0 & 1 & 0 & \cdots &0 & \alpha & 0 & c \\
	& 0& & & & 0 \\
	& \vdots & & & & \vdots \\
	& 0& &  & & 0\\
	0& 0&0 & \cdots& &1 & 0 & d  \\
	\end{block}
	\end{blockarray}
	\] where we may again use induction.
	\end{proof}
		
	The zigzag-compatible avatar of Proposition \ref{prop:gabcons} is given as follows. As before, we regard this result as the 'matrix version' of the decomposition theorem for zizag persistence modules.
	
	\begin{prop}
	   Given the sequence of matrices $A_\bullet$ as in \eqref{eq:zig}, 
there is a $g \in G$ such that $(gA)_\bullet$ has all its matrices in zizag barcode form.
	\end{prop}
	\begin{proof}
	We proceed by induction. The case $\ell=1$ is trivial. We can reduce ourselves by induction hypothesis to the case where the first $\ell-1$ matrices are in zizag barcode form. If $ V_{\ell-1} \to V_{\ell}$ we perform row  operations on $A_{\ell}$ to put it in reduced row echelon form. Otherwise, if $V_{\ell-1} \leftarrow V_{\ell}$, we perform column operations to put $A_{\ell}$ in reversed reduced column echelon form. That is, we put $A_{\ell}$ in the form 
	\[
\begin{blockarray}{cccccccccc}
\begin{block}{[cccccccccc]}
 & & &\vline \; \star& \star& \dots & \dots & \star & \\
 & 0& &\vline \; 1& 0& \dots & \dots & 0 & \\
 \cline{1-9}
 & & & & \ddots& & & & \\
 \cline{1-9}
  &  &  &  &  &\vline \; \star & \star &  \star& \\
  &   &  0&  &  & \vline \; 1& 0 & 0 & \\
 \cline{1-9}
 & & &  &  & & \vline \;\star & \star & \\
 &  & &0 & &  & \vline \; 1 & 0 & \\
 \cline{1-9}
 & & &  & & &  & \vline \; \star & \\
 & & & & 0& & & \vline \; 1& \\
 \cline{1-9}
 & & & & & & 0 & & \\
\end{block}
\end{blockarray}
\]

\noindent This is achieved through a slight tweak to the standard Gaussian algorithm for placing matrices in column echelon form, where one starts with the last row and works upwards. Finally, we may apply Lemma \ref{lemma:zmatop} to the non-zero $\star$ term of $A_{\ell}$ obtain the desired result.
\end{proof}

	\subsection{Barcode bases of zizag modules}
	In order to characterise the set of barcode bases of a zizag module, we must again attempt to characterise the set $\Stab(A_\bullet)$ corresponding to linear maps of a zizag module in a barcode basis. In Section \ref{sec:stabiliser}, we defined an order $\preceq$ on the set of intervals that allowed us to classify the stabiliser as 
	\[
	\Stab(A_\bullet) ~ \cong ~ \prod_{[i,j]}\GL(d_{ij};\mathbb{F})\times \hspace{-.2in} \prod_{[i_1,j_1] \precneq [i_2,j_2] } \hspace{-.2in} \text{\rm Mat}(d_{i_{1}j_{1} }\times d_{i_{2}j_{2}};\mathbb{F}).
	\]
	
	The $\text{\rm Mat}(d_{i_{1}j_{1} }\times d_{i_{2}j_{2}};\mathbb{F})$ came from the fact that in standard persistence, the interval module $\Int[i_1,j_1]$ can be mapped non-trivially to the interval $\Int[i_2,j_2]$ if and only if $[i_1,j_1] \preceq [i_2,j_2]$. In zizag persistence, this is no longer true. 
	
	\begin{ex}
	One checks that in the following scenario
	
	
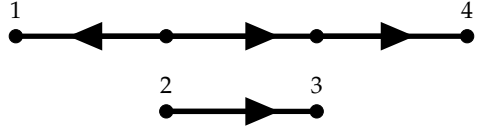
\begin{tikzpicture}[line cap=round,line join=round,>=triangle 45,x=1cm,y=1cm]
		\clip(-8.853170731707314,3) rectangle (8.551707317073165,5.56487804878048);
		\draw [line width=2pt] (-4,5)-- (-2,5);
		\draw [line width=2pt] (-2,5)-- (0,5);
		\draw [line width=2pt] (0,5)-- (2,5);
		\draw [line width=2pt] (-2,4)-- (0,4);
		\draw [->,line width=2pt] (-2,5) -- (-3.3,5);
		\draw [->,line width=2pt] (-2,5) -- (-0.5,5);
		\draw [->,line width=2pt] (-2,4) -- (-0.5,4);
		\draw [->,line width=2pt] (0,5) -- (1.3,5);

		\begin{scriptsize}
		\draw [fill=black] (-4,5) circle (2.5pt);
		\draw [color=black] (-4,5.35) node {1};
		\draw [fill=black] (-2,5) circle (2.5pt);
		\draw [fill=black] (0,5) circle (2.5pt);
		\draw [fill=black] (2,5) circle (2.5pt);
		\draw [color=black] (2,5.35) node {4};
		\draw [fill=black] (-2,4) circle (2.5pt);
		\draw [color=black] (-2,4.35) node {2};
		\draw [fill=black] (0,4) circle (2.5pt);
		\draw [color=black] (0,4.35) node {3};
		\end{scriptsize}
		\end{tikzpicture}
		
\noindent the interval $[1,4]$ may be non-trivially mapped to the interval $[2,3]$.
	\end{ex}
	
To then obtain a similar classification result for zizag modules, we see we must adapt our partial order $\preceq$ to the type $\tau$ of our zizag module. 

\begin{definition} Let $\tau$ be a type of zizag module, defining an orientation on the standard length $\ell$ quiver. Let $\preceq_{\tau}$ be the binary relation on $\set{[i,j] \in \Z^2 \mid  0 \leq i \leq j \leq \ell}$ given by
	\[[i_1,j_1] \preceq_\tau [i_2,j_2] \Leftrightarrow \begin{cases}
	[i_1,j_1] \cap [i_2,j_2]=[i,j] \neq \emptyset \quad \text{and}  \\
i_1 \leq i_2 \quad \text{if} \quad i-1 \rightarrow i, \quad i_2 \leq i_1 \quad \text{if} \quad i-1 \leftarrow i \quad \text{and}  \\
j_1 \leq j_2 \quad \text{if} \quad j \rightarrow j+1, \quad j_2 \leq j_1 \quad \text{if} \quad j \leftarrow j+1
	\end{cases}
	\]
\end{definition}
We observe that when all maps point forward, i.e., when $\tau=ff \ldots f$, we have $\preceq_\tau=\preceq$. And if all maps point backwards, ie $\tau= qq \ldots q$, then $\preceq_\tau$ is the reverse of $\preceq$ in the sense that 
\[
[i_1,j_1] \preceq_\tau [i_2,j_2] \Leftrightarrow [i_2,j_2] \preceq [i_1,j_1].
\]
Our description of the stabiliser from Theorem \ref{theorem:stabbij} relied on the compatibility of $\preceq$ with the lexicographical total order $\trianglelefteq$ on the bars. Having produced a $\tau$-analogue of $\preceq$, we must now construct the zigzag version of $\trianglelefteq$. To this end, we define two new auxiliary total orders on the set of all possible endpoints $\set {0, \ldots , \ell}$. Note that any such order $<$ amounts to a choice of element $\sigma_\ell$ lying in the permutation group $S_{\ell+1}$ via the identification
\[\sigma(0) < \sigma(1) < \cdots  < \sigma(\ell).
\]

\begin{definition} Let $\leq_\tau$ be the total order on $\set {0,1 \ldots \ell }$ given by the permutation $\sigma_\ell$ defined inductively as follows. Assuming we have ordered $\set {0, \ldots , i}$ with corresponding permutation $\sigma_i$, we order $\set {0, \ldots , i+1}$ with permutation
$\sigma_{i+1}$, setting 

\[\sigma_{i+1}(k)= \begin{cases} \sigma_{i}(k) & \text{if} \quad i \longrightarrow i+1 \quad \text{for} \quad k \in \set{0,\ldots i} \\
i+1 & \text{if} \quad  i \longrightarrow i+1 \quad \text{for} \quad k=i+1\\
\sigma_{i}(k)+1 & \text{if} \quad i \longleftarrow i+1 \quad \text{for} \quad k \in \set{0. \ldots i} \\
0 & \text{if} \quad i \longleftarrow i+1 \quad \text{for} \quad k=i+1
\end{cases}
\]
\end{definition}

\begin{definition}  Let ${\leq^*_\tau}$ be the total order on $\set {0,1 \ldots, \ell }$ given by the permutation $\sigma^*_\ell$ defined inductively as follows. Assuming we have ordered $\set {j+1, \ldots , \ell}$ with corresponding permutation $\sigma^*_{j+1}$, we order $\set{j, \ldots , \ell}$ with permutation
${\sigma}^*_{j}$, setting 
\[{\sigma}^*_{j}(k)= \begin{cases} {\sigma}^*_{j+1}(k) & \text{if} \quad j \longrightarrow j+1 \quad \text{for} \quad k \in \set{j+1. \ldots \ell} \\
j & \text{if} \quad  j \longrightarrow j+1 \quad \text{for} \quad k=j\\
 {\sigma}^*_{j+1}(k)-1 & \text{if} \quad j \longleftarrow j+1 \quad \text{for} \quad k \in \set{j+1. \ldots \ell} \\
\ell & \text{if} \quad j \longleftarrow j+1 \quad \text{for} \quad k=j
\end{cases}
\]
\end{definition}
If $\tau=ff\ldots f$, both of the above total orders coincide with the standard ordering $\leq$ on $\set{0,1 \ldots \ell}$.

\begin{ex}
Consider zigzag modules of length $4$ with type $0 \longleftarrow 1 \longrightarrow 2 \longleftarrow 3$. For such modules, the two total orders defined above are 
\begin{align*}
&3 \leq_\tau 1 \leq_\tau 0 \leq_\tau 2, \text{ and} \\
&1 \leq^*_\tau 3 \leq^*_\tau 2 \leq^*_\tau 0.
\end{align*}
\end{ex}

\begin{definition}
    Let $\tau$ be a type of length-$\ell$ zigzag module. The corresponding total order $\trianglelefteq_{\tau}$ on $\set{[i,j] \in \Z^2 \mid  0 \leq i \leq j \leq \ell}$ is given by
	 \begin{center}
	        $ [a,b] \trianglelefteq_\tau [c,d] \Leftrightarrow a <_\tau c $ or $a=c$ and $b  \; {\leq}^*_\tau \; d$.
	  \end{center}
\end{definition}

By construction, if $\tau=f f \ldots f$, this coincides with the lexicographic order $\trianglelefteq$ from Section \ref{sec:stabiliser}, and $[i_1,j_1] \preceq_\tau [i_2,j_2] \Rightarrow [i_1,j_1] \trianglelefteq_\tau [i_2,j_2]$. 
Using this total order on the bars, we may generalise our notion of \textbf{ordered barcode bases} to zizag modules, by ordering the bars with the order $\trianglelefteq_\tau$.
We now have all the necessary tools to adapt our results from Section \ref{sec:stabiliser} to the case of zizag modules. 

\begin{theorem}\label{theorem:main}
	For each pair $[i,j]$ in $\set{0,1,\ldots,\ell}$ with $i \leq j$, let $d_{ij}$ be the multiplicity in the barcode of $(V_\bullet,p_\bullet, \tau)$, with the understanding that $d_{ij} = 0$ whenever $[i,j]$ is not in $\Barc(V_\bullet,p_\bullet, \tau)$. Then there is a bijection of sets:
	\[
	\Stab(A_\bullet) ~ \cong ~ \prod_{[i,j]}\GL(d_{ij};\mathbb{F})\times \hspace{-.2in} \prod_{[i_1,j_1] \precneq_\tau [i_2,j_2]} \hspace{-.2in} \text{\rm Mat}(d_{i_{1}j_{1} }\times d_{i_{2}j_{2}};\mathbb{F}).
	\]
\end{theorem}

 \begin{prop} Given a zizag module $(V_\bullet, p_\bullet, \tau)$ together with an ordered barcode basis $\B$ with matrix representation $A_\bullet$, the set of all ordered barcode bases is given by the orbit $\Stab(A_\bullet) \B$.
    \end{prop}
 \subsection{Maps of zizag modules}  
We now turn our attention to morphisms $(V_\bullet, p_\bullet, \tau) \to (W_\bullet, \tilde{p}_\bullet, \tau)$ between zigzag persistence modules of the same type $\tau$. Each such morphism is determined by linear maps $\phi_i : V_i \mapsto W_i$ along with the requirement that the evident squares commute. That is, if $p_i=f_i$ (and so $\tilde{p}_i=\tilde{f}_i$), we have $\phi_i \circ f_i=\tilde{f}_i \circ \phi_{i-1}$, and if $p_i=q_i$ (and so $\tilde{p}_i=\tilde{q}_i)$, we have $\phi_{i-1
} \circ q_i=\tilde{q}_i \circ \phi_i$.  This is best represented through the following diagram
\[
\xymatrixcolsep{.5in}
\xymatrixrowsep{.4in}
\xymatrix{
V_0 \ar@{<->}[r]^{p_1} \ar@{->}[d]_{\phi_0} & V_1 \ar@{<->}[r]^{p_2} \ar@{->}[d]_{\phi_1} & \cdots \ar@{<->}[r]^{p_{\ell-1}}  & V_{\ell-1} \ar@{<->}[r]^{p_\ell} \ar@{->}[d]_{\phi_{\ell-1}}& V_\ell \ar@{->}[d]_{\phi_{\ell}} \\
W_0 \ar@{<->}[r]_{ \tilde{p}_1} & W_1 \ar@{<->}[r]_{ \tilde{p}_2} & \cdots \ar@{<->}[r]_{ \tilde{p}_{\ell-1}} & W_{\ell-1} \ar@{<->}[r]_{ \tilde{p}_\ell} & W_\ell
}
\]
where each natural square commutes. 

As in Section \ref{sec:maps}, such maps may be identified as ladder persistence module on the rectangle quiver with orientation $\tau$. We denote such modules $(V_\bullet, W_\bullet, \phi_\bullet, \tau)$. We may analogously define a special class of of modules, namely $\textbf{R}_{\tau_{[i_1,j_1]_\bullet}}^{ [i_2,j_2]}$, $\Int _\tau ^{+}[i,j]_\bullet$ and $\Int _\tau ^{-}[i,j]_\bullet$. Here is the zigzag analogue of Definition \ref{def:nested}
\begin{definition}
    Given a type $\tau$ of zizag module, we say a bar $[i_2,j_2]$ is strictly nested in a bar $[i_1,j_1]$ with regards to $\tau$, denoted $[i_2,j_2] \subset_\tau [i_1,j_1]$, if they are non-intersecting with $[i_1,j_1] \trianglelefteq_\tau [i_2,j_2]$ but $[i_1,j_1] \not \preceq_\tau [i_2,j_2]$.
\end{definition}

As in Section \ref{sec:maps}, maps of zizag module may be compactly represented as block upper triangular matrices using ordered barcode basis for the source and target and zizag modules. Having excluded strictly nested bars, we are free to perform operations of type $(1)$, $(2)$ and $(3)$ \`{a} la Theorem \ref{theorem:nonest} and obtain the following result.
 \begin{theorem}\label{theorem:mainII}
      Let $(V_\bullet, W_\bullet, \phi_\bullet, \tau)$ be a ladder persistence module of length $\ell+1$ and type $\tau$, where neither $V_\bullet$ nor $W_\bullet$ admit a pair of strictly nested bars with regards to $\tau$.   Then there are integers $r_{[i_1,j_1]}^{ [i_2,j_2]}$, $d_{ij}^{\pm} > 0$ for which :
      \[(V_\bullet, W_\bullet, \phi_\bullet, \tau ) \simeq \bigoplus_{[i_1,j_1] \preceq_\tau [i_2,j_2]} (\text{\bf R}_{\tau_{[i_1,j_1]_\bullet}}^{ [i_2,j_2]})^{r_{[i_1,j_1]}^{ [i_2,j_2]}} \oplus \bigoplus_{i \leq j} (\Int_\tau ^{+}[i,j]_\bullet)^{d_{ij}^{+}} \oplus \bigoplus_{i \leq j} (\Int_\tau ^{-}[i,j]_\bullet)^{d_{ij}^{-}}
      \]
  \end{theorem}

	\bibliographystyle{alpha}
	\bibliography{refs}

\end{document}